\documentclass[11pt,fleqn]{article}

\usepackage[paper=a4paper]
  {geometry}

\usepackage{paragraphs}
\usepackage{hyperref}
\usepackage{amsthm,thmtools}

\usepackage[utf8]{inputenc}
\usepackage[spanish,english]{babel}
\usepackage{enumitem}
\usepackage[osf,noBBpl]{mathpazo}
\usepackage[alphabetic,initials]{amsrefs}
\usepackage{amsfonts,amssymb,amsmath}
\usepackage{mathtools}
\usepackage{graphicx}
\usepackage[poly,arrow,curve,matrix]{xy}
\usepackage{wrapfig}
\usepackage{xcolor}
\usepackage{helvet}
\usepackage{stmaryrd}
\usepackage{tikz}

\declaretheoremstyle[headformat=swapnumber, spaceabove=\paraskip,
bodyfont=\itshape]{mystyle}
\declaretheoremstyle[headformat=swapnumber, spaceabove=\paraskip,
bodyfont=\normalfont]{mystyle-plain}

\declaretheorem[name=Definition, sibling=para, style=mystyle]{Definition}

\declaretheoremstyle[numbered=no, spaceabove=\paraskip,
bodyfont=\itshape]{mystyle-empty}
\declaretheoremstyle[numbered=no, spaceabove=\paraskip,
bodyfont=\itshape]{mystyle-empty-plain}
\declaretheorem[name=Lemma, style=mystyle-empty]{Lemma*}
\declaretheorem[name=Proposition, style=mystyle-empty]{Proposition*}
\declaretheorem[name=Theorem, style=mystyle-empty]{Theorem*}
\declaretheorem[name=Corollary, style=mystyle-empty]{Corollary*}
\declaretheorem[name=Definition, style=mystyle-empty]{Definition*}
\declaretheorem[name=Remark, style=mystyle-empty]{Remark*}
\declaretheorem[name=Example, style=mystyle-empty]{Example*}

\makeatletter
\renewenvironment{proof}[1][\textit{Proof}]{\par
  \pushQED{\qed}%
  \normalfont \topsep.75\paraskip\relax
  \trivlist
  \item[\hskip\labelsep
        \itshape
    #1\@addpunct{.}]\ignorespaces
}{%
  \popQED\endtrivlist\@endpefalse
}
\makeatother

\newcommand\NN{\mathbb N}
\newcommand\CC{\mathbb C}
\newcommand\QQ{\mathbb Q}
\newcommand\RR{\mathbb R}
\newcommand\ZZ{\mathbb Z}

\newcommand\ot{\otimes}
\renewcommand\to{\longrightarrow}
\renewcommand\phi{\varphi}
\newcommand\qbinom[2]{\genfrac{[}{]}{0pt}{0}{#1}{#2}}
\newcommand\vspan[1]{\left\langle #1 \right\rangle}

\newcommand\A{\mathcal A}
\newcommand\B{\mathcal B}

\renewcommand\O{\mathcal O}

\newcommand\F{\mathcal F}
\newcommand\GG{\Gamma}
\renewcommand\L{\mathcal L}

\newcommand\tphi{\tilde \phi}
\renewcommand\b{\mathbf b}
\renewcommand\k{\Bbbk}
\renewcommand\q{\mathbf{q}}
\newcommand\m{\mathfrak m}
\newcommand\n{\mathfrak n}
\newcommand\g{\mathfrak g}
\newcommand\opp{\circ}

\newcommand\join{\vee}
\newcommand\meet{\wedge}
\newcommand\lex{\mathsf{lex}}
\newcommand\str{\mathsf{str}}

\newcommand\flagbasis{FB}
\newcommand\schubertbasis{SB}

\DeclareMathOperator\Mod{\mathsf{Mod}}
\DeclareMathOperator\mmod{\mathsf{mod}}
\DeclareMathOperator\Hom{\mathsf{Hom}}
\DeclareMathOperator\Ext{\mathsf{Ext}}
\DeclareMathOperator\GrHom{\underline{\mathsf{Hom}}}
\DeclareMathOperator\GrExt{\underline{\mathsf{Ext}}}

\DeclareMathOperator\gr{\mathsf{gr}}
\DeclareMathOperator\im{Im}
\DeclareMathOperator\id{injdim}
\DeclareMathOperator\pd{pdim}
\DeclareMathOperator\ldim{ldim}
\DeclareMathOperator\depth{depth}
\DeclareMathOperator\lcd{lcd}
\DeclareMathOperator\Spec{Spec}
\DeclareMathOperator\Id{Id}
\DeclareMathOperator\rank{rk}
\DeclareMathOperator\rk{rk}

\DeclareMathOperator\relint{\mathsf{relint}}
\DeclareMathOperator\wt{\mathsf{wt}}

\title{
Quantum toric degeneration of quantum flag and Schubert varieties
}
\date{}
\author{L. Rigal, P. Zadunaisky}

\begin{document}
\maketitle

\begin{abstract}
We show that certain homological regularity properties of graded connected 
algebras, such as being AS-Gorenstein or AS-Cohen-Macaulay, can be tested by 
passing to associated graded rings. In the spirit of noncommutative algebraic 
geometry, this can be seen as an analogue of the classical result that, in a 
flat family of varieties over the affine line, regularity properties of 
the exceptional fiber extend to all fibers. 
We then show that quantized coordinate rings of flag varieties and Schubert 
varieties can be filtered so that the associated graded rings are twisted 
semigroup rings in the sense of \cite{RZ}. This is a noncommutative version
of the result due to Caldero \cite{C} stating that flag and Schubert varieties 
degenerate into toric varieties, and implies that quantized coordinate rings 
of flag and Schubert varieties are AS-Cohen-Macaulay.
\end{abstract}

\section{Introduction}
Let $\k$ be a field, and let $A$ be a noetherian commutative algebra over 
$\k$. If we put
an ascending filtration on $A$ then we can build the Rees ring of the 
filtration, which is a free $\k[t]$-algebra $\mathcal R$ such that $A \cong 
\mathcal R / (t - \lambda)\mathcal R$ for all $\lambda \in \k^\times$, while
$\mathcal R / t \mathcal R$ is isomorphic to the associated graded ring. In 
geometric terms, if $\k$ is algebraically closed then the variety associated 
to $\mathcal R$ is a flat family over the affine line, whose generic fiber is 
isomorphic to $\Spec A$ and whose fiber over $0$ is isomorphic to 
$\Spec~\gr~A$; in this context the fiber over $0$ is called a degeneration of 
$\Spec A$. A standard result from algebraic geometry states that if the 
fiber over $0$ is regular (resp. Gorenstein, Cohen-Macaulay, or any 
other of a long list of properties) then \emph{all} fibers are regular 
(resp. Gorenstein, Cohen-Macaulay, etc.)

Of course, the idea of studying a ring by imposing a filtration and passing to 
the associated graded ring is a basic tool in an algebraist's toolbox, and 
can be applied outside of a geometric context. In particular the hypothesis of 
commutativity is not necessary for filtered-to-graded methods to work. 
However, in the spirit of noncommutative algebraic geometry, we should look at 
the case where $A$ is noetherian, $\NN$-graded and connected (i.e. $A_0 = \k$) 
with an eye on the geometric case. Although in this case there are no 
varieties associated to our algebras as in the commutative setting, we have 
suitable analogues of the notions of being 
regular, or Gorenstein, or Cohen-Macaulay, defined in purely homological terms 
(see paragraph \ref{AS-reg}). Hence it makes sense to ask whether these 
properties are ``stable by flat deformation'', i.e. if the fact that $\gr A$ 
has any of these properties implies that $A$ also has that property. The 
objective of this paper is to develop these ideas in order to study natural 
classes of noncommutative algebraic varieties, and its first main result 
\ref{transfer-of-regularity} shows that indeed, good geometric properties are
stable by degeneration in this context.


In the commutative setting, the usual flag variety and its Schubert 
subvarieties are examples where the degeneration method is successful. 
The same holds for the more general flag and Schubert varieties that one
may associate to a semisimple Lie group. On the other hand, the theory of 
quantum groups provides natural quantum analogues of flag and Schubert 
varieties, whose classical counterparts can be recovered as semiclassical 
limits when the deformation parameter tends to $1$. This is the class of 
noncommutative varieties we intend to study.

The degeneration approach to the study of flag and Schubert varieties was 
pioneered by 
de Concini, Eisenbud and Procesi \cite{DCEP82}, and followed by Gonciulea and
Lakshmibai \cite{GL}, and others. Caldero was the first to prove that any
Schubert variety of an arbitrary flag variety degenerates to an \emph{affine 
toric} variety in \cite{C}. A degeneration to a toric variety is particularly 
convenient, since the geometric properties of toric varieties are easily 
tractable, being encoded in the combinatorial properties of a semigroup 
naturally attached to them in the affine case. For a general survey on the 
subject of toric degenerations, including recent results, we refer the reader 
to \cite{FFL}. 

It is striking that even though he is interested in the classical objects, 
in \cite{C} Caldero uses the theory of quantum groups (more precisely the 
global bases of Luzstig and Kashiwara) to produce bases of the coordinate 
rings of the classical objects. His main argument then relies on Littelmann's 
string parametrization of the crystal basis of the negative part of the 
quantized enveloping algebra, which allows him to show that his bases have 
good multiplicative properties. It follows that one can build out of them
adequate filtrations of the coordinate rings, leading to a toric degeneration.

In contrast, we wish to produce degenerations of quantum Schubert varieties
at the noncommutative level. This led us to introduce noncommutative analogues
of affine toric varieties as a suitable target for degeneration and to study 
their geometric properties in \cite{RZ2}, where we show that just as in the 
classical case, these properties are encoded in the associated semigroup. 
Inspired by Caldero's work, in this article we show that quantum Schubert 
varieties degenerate into quantum toric varieties and explore the consequences 
of this fact. 
\\

We now discuss the contents of each section.

Section \ref{graded-connected-algebras} begins with a recollection of general 
material on connected $\NN^r$-graded algebras, their homological regularity 
properties, and the behavior of these properties with respect to a change of 
grading. In the last subsection we prove our main transfer result: if a 
connected $\NN^r$-graded algebra has a filtration by finite dimensional graded 
subspaces, then (under a mild technical condition known as property $\chi$) 
the original algebra inherits the regularity properties of the associated 
graded algebra, see Theorem \ref{transfer}. The key point is the existence of 
a spectral sequence relating the $\Ext$ spaces between an algebra and
its associated graded algebra; property $\chi$ guarantees that these spaces
are finite dimensional, which is crucial for the transfer of properties.

The aim of section \ref{qatv} is to introduce a general framework allowing to 
establish the existence of a noncommutative toric degeneration. We first
review the basic theory of affine semigroups and 
the regularity properties of quantum affine toric varieties from \cite{RZ2}. 
We then introduce the notion of an algebra of $(S,\phi)$-type, where $S$
is an affine semigroup $S \subset \NN^r$ and $\phi: S \to \NN$ is a semigroup 
morphism. Any positive affine semigroup has a unique minimal 
presentation as a quotient of a finitely generated free noncommutative monoid. 
An algebra of $(S,\phi)$-type is an algebra with a presentation modeled after
the presentation of $S$, with as many generators as $S$ and whose product
reflects that of $S$ up to terms of smaller order with respect to
the order induced on $S$ by $\phi$. 
We then show that an algebra of $(S,\phi)$-type degenerates to a quantum 
toric variety if and only if it has a basis of monomials indexed by $S$. 
At this stage it is useful to recall the notion of an Algebra with a 
Straightening Law (see \cite{DCEP82} in the commutative case and \cite{RZ}
in the quantum case). It turns out that, under a mild assumption, this 
structure provides a natural toric degeneration.

Schubert varieties in the grassmannian enjoy such an ASL structure, both in he 
classical and quantum case. However, in more general Schubert varieties such a 
structure no longer exists. The notion of an algebra of $(S,\phi)$-type is a 
generalization of the ASL structure which does apply in this more general 
context, and the aforementioned basis indexed by $S$ plays the role of the
standard monomial basis. For more details regarding these two
notions see \ref{s-phi-type}. 
We finish this section by introducing the notion of an $(S,<)$-basis, where
$S$ is a semigroup and $<$ is a total order on $S$. It turns out that the 
existence of such a structure associated to an affine semigroup $S$, with the 
order obtained by pulling back the lexicographic order of $\NN^t$ through an 
embedding $S \hookrightarrow \NN^t$, is enough to obtain an $(S,\phi)$-type
structure and to prove the existence of a toric degeneration.

Section \ref{q-flag-schubert} contains a review of the definitions of 
quantum flag varieties and their Schubert subvarieties. We then introduce bases
for these objects originally defined by Caldero, who showed that they are
parametrized by affine semigroups using a variant of Littelmann's string 
parametrization, and that they have the multiplicative properties needed to
produce a toric degeneration. More precisely, we show that these bases are
$(S,<)$-bases for an adequate semigroup $S$. We stress that, since Caldero was 
interested 
only in classical varieties, he worked over the field $\CC(q)$, where $q$ is a 
transcendental variable, and with an enlarged version of the quantum group
$U_q(\g)$ (which Janzen calls the adjoint type). Since we are interested in 
the quantum setting, we work over the ring $\ZZ[v,v^{-1}]$ and eventually 
specialize to an arbitrary field containing a deformation parameter $q$ which 
is not a root of unity. Furthermore we use the small (simply connected) 
version of $U_q(\g)$. 

While technical, we felt that these differences merited a detailed exposition 
of the arguments, which we present in paragraphs \ref{T:A-crystal-basis} to 
\ref{Schubert-deg}. Our proof that Caldero's bases have the desired 
multiplicative properties is different that the one found in \cite{C}. While 
he uses multiplicative properties of the dual global basis, we analize the
coproduct of $U^-_q(\g)$ in terms of the global basis. To do this we are
led to study a second basis introduced by 
Littelman in \cite{Lit}, consisting of certain monomials of divided 
differences. This monomial basis seems to be better adapted for
computations, and has been used to obtain more general degeneration results
in the classical case by Fang, Fourier and Littelmann \cite{FFL2}. However, 
the global basis, or rather its string parametrization, is crucial in 
order to prove that one obtains a \emph{toric} degeneration.

\noindent\textbf{Acknowledgements:} The second author would like to thank
Xin Fang for an illuminating dicussion on the relation between monomial and
canonical bases, and Köln University for its hospitality.

\section{Degeneration of graded connected algebras}
\label{graded-connected-algebras}
Let $r$ be a positive integer. Throughout this section $A$ denotes a 
noetherian connected $\NN^{r}$-graded $\k$-algebra. Here connected means that 
the homogeneous component of $A$ of degree $(0, \ldots, 0)$ is isomorphic to 
$\k$ as a ring, so the ideal generated by all homogeneous elements of non-zero 
degree is the unique maximal graded ideal of $A$; we denote this ideal by 
$\m$. Clearly $A/\m \cong \k$ as vector spaces, and whenever we consider $\k$ 
as $A$-bimodule, it will be with this structure.

\subsection*{Graded modules}

\paragraph
\label{grmod-generalities}
We denote by $\ZZ^{r}\Mod A$ the category of $\ZZ^{r}$-graded $A$-modules with 
homogeneous morphisms of degree $0$. We review some general properties of this 
category; the reader is referred to \cite{NV}*{chapter 2} for proofs and 
details.

The category $\ZZ^r\Mod A$ has enough projectives and injectives, so we may 
speak of the graded projective and injective dimensions of an object $M$, 
which we denote by $\pd_A^{\ZZ^r} M$ and  $\id^{\ZZ^r}_A M$, respectively. We 
denote by $\ZZ^{r}\mmod A$ the subcategory of finitely generated 
$\ZZ^{r}$-graded $A$-modules. Since $A$ is noetherian $\ZZ^r \mmod A$ is an 
abelian category with enough projectives. 

For every object $M$ of $\ZZ^{r}\Mod A$ and every $\xi \in \ZZ^{r}$ we denote 
by $M_\xi$ the homogeneous component of $M$ of degree $\xi$. Also, we denote 
by $M[\xi]$ the object of $\ZZ^{r}\Mod A$ with the same underlying $A$-module 
as $M$ and with homogeneous components $M[\xi]_{\zeta} = M_{\zeta + \xi}$ for 
all $\zeta \in \ZZ^{r}$. If $f: N \to M$ is a morphism in $\ZZ^{r}\Mod A$ then 
the same function defines a morphism $f[\xi]:N[\xi] \to M[\xi]$. In this way 
we get an endofunctor $[\xi]: \ZZ^{r}\Mod A \to \ZZ^{r}\Mod A$, called the 
$\xi$-suspension functor; it is an autoequivalence, with inverse $[-\xi]$.

Given $N, M$ objects of $\ZZ^{r}\Mod A$ we set 
\begin{align*} 
  \GrHom_A(N, M) = \bigoplus_{\xi \in \ZZ^r} \Hom_{\ZZ^r\Mod A}(N,M[\xi]).  
\end{align*} 
This is a $\ZZ^{r}$-graded vector space, with its component of degree $\xi \in 
\ZZ^{r}$ equal to the space of homogeneous $A$-linear maps of degree $\xi$ 
from $N$ to $M$. For every $i \geq 0$ we denote by $\GrExt_A^i$ the $i$-th 
right derived functor of $\GrHom_A$. We point out that 
\[
  \GrHom_A(N[\xi], M) \cong \GrHom_A(N,M[-\xi]) \cong \GrHom_A(N,M)[-\xi],
\]
as $\ZZ^r$-graded vector spaces, and that these isomorphisms induce analogous 
ones for the corresponding right derived functors.

\paragraph
\label{ext-change-of-grading}
Given $\xi = (\xi_1, \ldots, \xi_r) \in \ZZ^r$ we set $|\xi| = \xi_1 + \cdots +
\xi_r$. Given a $\ZZ^{r}$-graded vector space $V$, we denote by $|V|$ the 
$\ZZ$-graded vector space whose $n$-th homogeneous component is 
\begin{align*}
  |V|_n = \bigoplus_{|\xi| = n} V_\xi.
\end{align*}
In particular $|A|$ is a connected $\NN$-graded algebra. Also, if $M$ is a 
$\ZZ^r$-graded $A$-module then $|M|$ is a $\ZZ$-graded $|A|$-module, and this 
assignation is functorial. Since $A$ is noetherian, \cite{RZ2}*{Proposition 
1.3.7} implies that for every $i \geq 0$ and any pair of $\ZZ^r$-graded 
modules $N, M$, with $N$ finitely generated, there is an isomorphism of 
$\ZZ$-graded modules
\[
  |\GrExt^i_A(N,M)| \cong \GrExt^i_{|A|}(|N|,|M|),
\]
natural in both variables. 

\subsection*{Homological regularity properties}
In this subsection we discuss some homological properties a connected 
$\NN^r$-graded algebra may posses. Most of the material found in this section 
is standard for connected $\NN$-graded algebras.

\paragraph
\label{chi}
Let $M$ be a $\ZZ^r$-graded $A$-module. We say that $\chi(M)$ holds if for 
each $i \geq 0$ the graded vector space $\GrExt_A^i(\k, M)$ is finite 
dimensional, and say that the algebra $A$ has property $\chi$ if $\chi(M)$ 
holds for every finitely generated $\ZZ^r$-graded $A$-module $M$. Property 
$\chi$ was originally introduced in \cite{AZ}*{section 3} and plays a 
fundamental role in noncommutative algebraic geometry.

\paragraph
\label{torsion-functor}
Associated to $A$ and $\m$ there is a \emph{torsion functor} 
\begin{align*} 
  \GG_\m: \ZZ^{r}\Mod A &\to \ZZ^{r}\Mod A \\ 
  M &\longmapsto \{x \in M \mid \m^n x = 0 \mbox{ for } n \gg 0\}, 
\end{align*} 
which acts on morphisms by restriction and correstriction. The torsion functor 
is left exact, and for each $i \geq 0$ its $i$-th right derived functor is 
denoted by $H_\m^i$ and called the \emph{$i$-th local cohomology functor} of 
$A$. 

There exists a natural isomorphism
\begin{align*} 
\GG_\m \cong \varinjlim_n \GrHom_A(A/\m^n, -) 
\end{align*}
which by standard homological algebra extends to natural isomorphisms
\begin{align*} 
H^i_\m \cong \varinjlim_n \GrExt^i_A(A/\m^n, -) 
\end{align*} 
for all $i \geq 1$. The proof of this fact is completely analogous to the 
one found in \cite{BS}*{Theorem 1.3.8} for commutative ungraded algebras.

We denote by $A^\opp$ the opposite algebra of $A$, which is also
a connected $\NN^r$-graded algebra, and by $\m^\opp$ its maximal graded ideal. 
We write $\GG_{\m^\opp}$ and $H^i_{\m^\opp}$ for the corresponding torsion and 
local cohomology functors, respectively.

\paragraph
\label{depth-and-ldim}
Given an object $M$ of $\ZZ^{r}\Mod A$, its \emph{depth} and \emph{local 
dimension} are defined as
\begin{align*} 
  \depth_\m M &= \inf \{i \in \NN \mid \GrExt^i_A(\k, M) \neq 0\}, \\ 
  \ldim_\m M &= \sup \{i \in \NN \mid H^i_\m(M) \neq 0\},
\end{align*} 
respectively. The \emph{local cohomological dimension} of $A$, denoted by 
$\lcd_\m A$, is the supremum of the $\ldim_\m M$ with $M$ finitely generated.

\paragraph
\label{AS-reg}
The following definition is taken from \cite{RZ2}*{Definition 2.1.1}. It is an 
$\NN^r$-graded analogue of the definition of the AS-Cohen-Macaulay, 
AS-Gorenstein and AS-regular properties for connected $\NN$-graded algebras 
found in the literature, see for example the introduction to \cite{JZ}.
\begin{Definition*}
Let $A$ be a connected noetherian $\NN^{r}$-graded algebra.  
\begin{enumerate} 
\item $A$ is called \emph{AS-Cohen-Macaulay} if there exists $n \in \NN$ such 
that $H^i_{\m}(A) = 0$ and $H_{\m^\opp}^i(A) = 0$ for all $i \neq n$.

\item $A$ is called \emph{left AS-Gorenstein} if it has finite graded 
injective dimension $n$ and there exists $\ell \in \ZZ^{r}$, called the 
\emph{Gorenstein shift} of $A$, such that
\[ 
\GrExt_{A}^i(\k,A) \cong 
  \begin{cases} 
    \k[\ell] & \mbox{for } i = n, 
    \\ 0 & \mbox{for } i \neq n,
  \end{cases} 
\] 
as $\ZZ^{r}$-graded $A^\opp$-modules. We say $A$ is \emph{right} AS-Gorenstein 
if $A^\opp$ is left AS-Gorenstein. Finally $A$ is \emph{AS-Gorenstein} if $A$ 
and $A^\opp$ are left AS-Gorenstein, with the same injective dimensions and 
Gorenstein shifts.

\item $A$ is called \emph{AS-regular} if it is AS-Gorenstein, and its left and 
right graded global dimensions are finite and equal.
\end{enumerate} 
\end{Definition*} 

\paragraph
\label{invariance-by-grading}
The properties discussed in paragraphs \ref{chi} to \ref{AS-reg} are defined 
in terms of the category of $\ZZ^r$-graded $A$-modules. In 
\ref{ext-change-of-grading} we defined the algebra $|A|$, which is equal to $A$
as algebra but is endowed with a connected $\NN$-grading induced by the 
grading on $A$ and the group morphism $|\cdot|$. The maximal graded ideals of 
$A$ and $|A|$ coincide as vector spaces, so we may ask whether the fact that 
$A$ has property $\chi$, or finite local dimension, or the AS-Cohen-Macaulay 
property, etc., implies that $|A|$ has the corresponding property.

Let us say that a property $P$ does not depend on the grading of $A$ if the 
following holds: for every connected $\NN^t$-graded algebra $B$ with maximal 
ideal $\n$, which is isomorphic to $A$ as algebra through an isomorphism that 
sends $\n$ to $\m$, $A$ has property $P$ if and only if $B$ has property $P$. 
As shown in \cite{RZ2}*{Corollary 1.3.9}, the local dimension of $A$ does not 
depend on the grading of $A$. An analogous result is proved in 
\cite{RZ2}*{Remark 2.1.7} for the properties defined in \ref{AS-reg}. However, 
the situation is more delicate for property $\chi$. The following lemma shows 
that property $\chi$ is independent of the grading of $A$ under the hypothesis 
that $A$ has finite local dimension; we do not know whether this hypothesis 
can be eliminated.

\begin{Lemma*}
Suppose $\lcd_\m A < \infty$. Then the algebra $A$ has property $\chi$ if and 
only if $\chi(A)$ holds.
\end{Lemma*}
\begin{proof}
If $A$ has property $\chi$ then clearly $\chi(A)$ holds. To prove the opposite
implication, assume $\chi(A)$ holds. Recall from \ref{ext-change-of-grading} 
that for every $\ZZ^r$-graded $A$-module $M$ and for every $i \geq 0$ there 
exists a graded vector space isomorphism
\[
  |\GrExt^i_A(\k, M)|
    \cong \GrExt_{|A|}^i(|\k|, |M|),
\]
so $\chi(M)$ holds if and only if $\chi(|M|)$ holds; in view of this, the 
hypothesis implies $\chi(|A|)$ holds. 

Since $\ldim_{|\m|} |A| = \ldim_\m A < \infty$, we may apply 
\cite{RZ2}*{Proposition 2.2.6} and conclude that $|A|$ has property $\chi$. 
From this it follows that $\chi(|M|)$ holds for every $\ZZ^r$-graded 
$A$-module $M$, and hence so does $\chi(M)$.
\end{proof}

\paragraph
\label{chi-and-local-cohomology}
We finish this subsection with a technical result on the relation between 
property $\chi$ and local cohomology. 

\begin{Lemma*}
For every $n \in \NN$, let $A_{\geq n}$ be the ideal generated by all 
homogeneous elements of degree $\xi$ with $|\xi| \geq n$. Let $M$ be a 
finitely generated $\ZZ^r$-graded $A$-module such that $\chi(M)$ holds. Then 
for every $i \geq 0$ and every $t \in \ZZ$ there exists $n_0 \in \ZZ$ such 
that 
\begin{align*}
  \GrExt^i_A(A/A_{\geq n}, M)_\xi &\cong H^i_\m(M)_\xi
\end{align*}
for all $n \geq n_0$ and all $\xi \in \ZZ^r$ such that $|\xi| \geq t$.
\end{Lemma*}
\begin{proof}
Since $\m$ is finitely generated, say by elements $x_1, \ldots, x_r$ with 
degrees $\xi_1, \ldots, \xi_r$ such that $|\xi_i| \geq 1$, clearly $\m^n 
\subset A_{\geq n}$. Setting $l = \max \{|\xi_i| : 1 \leq i \leq r\}$ we 
obtain $A_{\geq ln} \subset \m^n$. Knowing this, the proof of 
\cite{BS}*{Proposition 3.1.1} easily adapts to show that for every $i \geq 0$ 
there exist natural isomorphisms
\begin{align*}
\varinjlim_n \GrExt^i_A(A/A_{\geq n}, -) 
  \cong \varinjlim_n \GrExt^i_A(A/\m^n, -) 
  \cong H^i_\m.
\end{align*} 
The statement of the lemma will follow if we show that for all $\xi$ as in the
statement, the homogeneous component of degree $\xi$ of the natural map
\begin{align*}
  \pi^n: \GrExt^i_A(A/A_{\geq n}, M) \to H^i_\m(M)
\end{align*}
is an isomorphism for $n \gg 0$.

Fixing $t$ as in the statement, \cite{AZ}*{Proposition 3.5 (1)} implies that
the natural map $\pi^n_d: \GrExt^i_{|A|}(|A|/|A_{\geq n}|, |M|)_d \to 
H^i_{|\m|}(|M|)_d$ is an isomorphism for all $d \geq t$ if $n$ is large 
enough. Now by \cite{RZ2}*{Propositions 1.3.7 and 1.3.8} there exist 
isomorphisms
\begin{align*}
\GrExt^i_{|A|}(|A|/|A_{\geq n}|, |M|)_d 
&\cong \bigoplus_{|\xi| = d}\GrExt^i_{A}(A/A_{\geq n}, M)_\xi;\\
H^i_{|\m|}(|M|)_d 
  &\cong \bigoplus_{|\xi| = d} H^i_\m(M)_\xi.
\end{align*}
Since the assignation $M \mapsto |M|$ is functorial, we also get that 
$\pi^n_d = \bigoplus_{|\xi| = d} \pi^n_\xi$. Thus for all $\xi$ such
that $|\xi| \geq t$, the map $\pi^n_\xi$ is an isomorphism if $n$ is large 
enough.
\end{proof}

\subsection*{Transfer of regularity properties by degeneration}
\label{transfer-of-regularity}
In this subsection we prove that if $A$ has a filtration compatible with its 
grading, and the associated graded algebra has property $\chi$, then the 
regularity properties discussed in the previous subsection transfer from 
$\gr A$ to $A$. All undefined terms regarding filtrations can be found in 
\cite{VO}*{chapter I}.

\paragraph
\label{regular-filtration}
Recall that $A$ denotes a noetherian $\NN^r$-graded algebra. The general setup 
for the subsection is as follows: we assume that $A$ has a \emph{connected} 
filtration, that is an exhaustive filtration $\F = \{F_nA\}_{n \geq 
0}$, with $\k = F_0A \subset F_1A \subset \cdots \subset F_nA \subset \cdots 
\bigcup_{n \geq 0} F_nA = A$, such that each layer $F_nA$ is a finite 
dimensional graded vector space, and $F_nA \cdot F_m A \subset F_{n+m}A$ for 
all $n, m \in \NN$. For each $\xi \in \NN^r$ the homogeneous component $A_\xi$ 
has an induced filtration $\{F_nA_\xi\}_{n \geq 0}$, where $F_nA_\xi = F_nA 
\cap A_\xi$. Since $A_\xi$ is finite dimensional this filtration is finite, so 
the associated graded ring $\gr A$ is a connected and locally finite 
$\NN^{r+1}$-graded algebra. 

Given any $\ZZ^r$-graded $A$-module $M$ with a filtration whose layers are 
$\ZZ^r$-graded subspaces, we can construct the $\ZZ^{r+1}$-graded 
$\gr A$-module $\gr M$. If $M$ is any $\ZZ^r$-graded $A$-module then it can be 
endowed with such a filtration as follows: fix a graded subspace $N \subset M$ 
that generates $M$ over $A$, and for each $n \geq 0$ set $F_nM = (F_nA) N$. 
Any such filtration is called \emph{standard}, and is an exhaustive and 
discrete filtration by graded subspaces. If $M$ is finitely generated and $N$ 
is finite dimensional then the layers of this filtration are also finite 
dimensional. 

\paragraph
\label{P:ext-ss} 
The main tool used to transfer homological information from $\gr A$ to $A$ is 
a spectral sequence that we associate to any pair of $\ZZ^r$-graded 
$A$-modules $N, M$, which converges to $\GrExt^i_A(N,M)$ and whose first page 
consists of the homogeneous components of $\GrExt^i_{\gr A}(\gr N, \gr M)$. 
The proof is straightforward, but relies on several graded analogues of 
classical constructions for filtered rings. These constructions can be found 
in \cite{VO}*{Chapter I} and \cite{MR}*{Section 7.6}, and the proofs found in 
the references easily adapt to the graded context, so we use them without 
further comment.

In order to keep track of the extra component in the grading when passing to 
associated graded objects, we make a slight abuse of notation: given a 
$\ZZ^{r+1}$-graded vector space $V$, we denote by $V_{(\xi, p)}$ its 
homogeneous component of degree $(\xi_1, \ldots, \xi_r, p)$.

\begin{Proposition*} 
Let $A$ be an $\NN^r$-graded algebra with a connected filtration, and 
assume $\gr A$ is noetherian. Let $M, N$ be filtered $\ZZ^r$-graded 
$A$-modules, with $N$ finitely generated, and suppose that the filtration on 
$N$ is standard and the filtration on $M$ is discrete. Then
for every $\xi \in \ZZ^r$ there exists a convergent spectral sequence
\begin{align*} 
  E(N,M)_\xi: E_{p,q}^1 = \GrExt_{\gr A}^{-p-q}(\gr N, \gr M)_{(\xi,p)}
    &\Rightarrow \GrExt_A^{-p-q}(N,M)_\xi 
    &p,q \in \ZZ, 
\end{align*} 
such that the filtration of the vector spaces on the right hand side is finite.
\end{Proposition*}
\begin{proof}
By the $\ZZ^r$-graded version of \cite{MR}*{Theorem 6.17}, there exists a 
projective resolution $P^\bullet \to N$ by filtered projective $\ZZ^r$-graded 
$A$-modules with filtered differentials, such that the associated graded 
complex $\gr P^\bullet \to \gr N$ is a $\ZZ^{r+1}$-graded projective 
resolution of the $\gr A$-module $\gr N$. Using the filtration for the $\GrHom$
spaces defined in \cite{VO}*{section I.2}, the complex $\GrHom_A(P^\bullet, M)$
is a graded complex with a filtration by graded subcomplexes, whose 
differentials are filtered maps.

If we fix $\xi \in \ZZ^r$, the homogeneous component $\GrHom_A(P^\bullet, 
M)_\xi$ is a complex of filtered finite dimensional vector spaces. By 
\cite{W}*{5.5.1.2} there exists a spectral sequence with page one equal to
\begin{align*} 
  E_{p,q}^1 
  &= H_{p+q}\left(\frac{F_p \GrHom_A(P^\bullet,M)_\xi}
      {F_{p-1}  \GrHom_A(P^\bullet,M)_\xi}\right) 
    & p,q \in \ZZ,
\end{align*} 
that converges to   
\begin{align*} 
  H_{p+q}(\GrHom_{A}(P^\bullet,M)_\xi) \cong \GrExt^{-p-q}_A(N,M)_\xi.
\end{align*}
This last space is finite dimensional, and hence the filtration on it is 
finite. Thus we only need to prove that for each $p,q \in \ZZ$ there exists an 
isomorphism
\begin{align*} 
  H_{p+q}
  \left(
    \frac{F_p \GrHom_A(P^\bullet,M)_\xi}{F_{p-1} \GrHom_A(P^\bullet,M)_\xi}
  \right) 
    &\cong \GrExt_{\gr A}^{-p-q}(\gr N, \gr M)_{(\xi,p)}.
\end{align*} 
By \cite{VO}*{Lemma 6.4}, there exists an isomorphism of complexes 
\begin{align*}
\phi(P^\bullet, M): 
	\gr (\GrHom_A(P^\bullet, M)) 
		\to \GrHom_{\gr A}(\gr P^\bullet, \gr M),
\end{align*}
which is defined explicitly in the reference. Direct 
inspection shows that the map $\phi(P^\bullet, M)$ is homogeneous, so looking 
at its component of degree $\xi$ we obtain an isomorphism
\begin{align*}
\frac{F_p \GrHom_A(P^\bullet,M)_\xi}{F_{p-1} \GrHom_A(P^\bullet,M)}_\xi 
  \cong \GrHom_{\gr A}(\gr P^\bullet, \gr M)_{(\xi,p)}.
\end{align*}
Since $\gr P^\bullet$ is a $\ZZ^{r+1}$-graded projective resolution of $\gr 
N$, we obtain the desired isomorphism by applying $H_{p+q}$ to both sides of 
the isomorphism.
\end{proof}

The following Corollary is an immediate consequence of the previous 
Proposition.
\begin{Corollary*}
Let $A$ be an $\NN^r$-graded algebra with a connected filtration, and 
assume $\gr A$ is noetherian. Let $M, N$ be filtered $\ZZ^r$-graded 
$A$-modules, with $N$ finitely generated, and suppose that the filtration on 
$N$ is standard and the filtration on $M$ is discrete. Then the following hold.
\begin{enumerate}[label=(\alph*)]
\item \label{I:dims} For each $i \geq 0$ and each $\xi \in \ZZ^{r}$
\[
 \dim_\k \GrExt^i_A(N,M)_\xi 
   \leq \sum_{p=-\infty}^\infty 
    \dim_\k \GrExt^i_{\gr A}(\gr N, \gr M)_{(\xi, p)}.
\]
  
\item \label{I:pd-id} $\pd_A^{\ZZ^r} N \leq \pd_{\gr A}^{\ZZ^{r+1}} \gr N$ 
and $\id_A^{\ZZ^r} M \leq \id_{\gr A}^{\ZZ^{r+1}} \gr M$. 

\item \label{I:chi} If $\chi(\gr M)$ holds then $\chi(M)$ holds.
\end{enumerate}
\end{Corollary*}

\paragraph
\label{ldim} 
We now prove a result that relates the local cohomology of a $\ZZ^r$-graded 
$A$-module $M$ with that of its associated graded module $\gr M$. We will do 
this by combining Proposition \ref{P:ext-ss} with Lemma 
\ref{chi-and-local-cohomology} and the formalism of the change of grading 
functors introduced in \cite{RZ}*{Section 1.3}. We recall the
relevant details. Given a group morphism $\phi: \ZZ^r \to \ZZ^t$, with $t \in 
\NN$, there exists a functor $\phi_!: \ZZ^r\Mod \k \to \ZZ^t \Mod \k$ that 
sends a $\ZZ^r$-graded vector space $V$ to the $\ZZ^t$-graded vector space 
$\phi_!(V)$, whose homogeneous component of degree $\zeta \in \ZZ^t$ is
\[
  V_\zeta = \bigoplus_{\phi(\xi) = \zeta} V_\xi.
\]
Notice that $\phi_!$ does not change the underlying vector space of its 
argument, only its grading. If $R$ is a $\ZZ^r$-graded algebra then $\phi_!(R)$
is a $\ZZ^t$-graded algebra, and if $M$ is a $\ZZ^r$-graded $R$-module then 
$\phi_!(M)$ is a $\ZZ^t$-graded $\phi_!(R)$-module with the same underlying 
$R$-module structure as $M$.

\begin{Corollary*}
Let $A$ be an $\NN^r$-graded algebra with a connected filtration, and 
assume $\gr A$ is noetherian. Let $M$ be a filtered $\ZZ^r$-graded $A$-module 
with a discrete filtration, and assume $\chi(\gr M)$ holds. Then for each 
$i \geq 0$ and each $\xi \in \ZZ^r$ 
\[
  \dim_\k H^i_\m(M)_\xi 
    \leq \sum_{p=- \infty}^\infty \dim_\k H^i_{\gr \m}(\gr M)_{(\xi, p)}.
\]
\end{Corollary*}
\begin{proof}
Let $\pi: \ZZ^{r+1} \to \ZZ^r$ be the projection to the first $r$-coordinates, 
and set $B = \pi_!(\gr A)$, so for every $\xi \in \ZZ^r$
\begin{align*}
B_\xi = \bigoplus_{p \in \ZZ} \gr A_{(\xi, p)} = \gr (A_\xi).
\end{align*}
Thus $B$ is a connected $\NN^r$-graded algebra. Set $\n = \pi_!(\gr \m)$, 
which is the maximal graded ideal of $B$, and set $\tilde M = \pi_!(\gr M)$.

By \cite{RZ2}*{Proposition 1.3.7} $\chi(\gr M)$ implies
$\chi(\tilde M)$ and by item \ref{I:chi} of Corollary \ref{P:ext-ss} it 
also implies $\chi(M)$, so we may apply Lemma \ref{chi-and-local-cohomology} 
to $M$ and $\tilde M$, and deduce that for given $\xi \in \ZZ^r$ and $i \geq 0$
the natural maps 
\begin{align*}
  \GrExt^i_A(A/A_{\geq n}, M)_\xi &\to H^i_\m(M)_\xi\\
  \GrExt^i_B(B/B_{\geq n}, \tilde M)_\xi &\to H^i_\n( \tilde M)_\xi
\end{align*}
are isomorphisms for $n \gg 0$.

Combining this with \cite{RZ2}*{Propositions 1.3.7 and 1.3.8}, we obtain a 
chain of isomorphisms
\begin{align*}
\bigoplus_{p \in \ZZ} 
  \GrExt^i_{\gr A}&\left(\frac{\gr A}{\gr (A_{\geq n})}, 
    \gr M \right)_{(\xi, p)} \\
  &\cong 
    \GrExt^i_B \left(B/B_{\geq n}, \tilde M \right)_\xi 
  \cong 
    H^i_\n(\tilde M)_\xi 
  \cong 
    \bigoplus_{p \in \ZZ} H^i_{\gr \m}(\gr M)_{(\xi, p)} 
\end{align*}
By definition $\gr(A/A_{\geq n}) \cong \gr A / \gr (A_{\geq n})$, so applying 
item  \ref{I:dims} of Corollary \ref{P:ext-ss} and taking $n \gg 0$ we obtain
\begin{align*}
\dim_\k H^i_\m(M)_\xi 
  &= \dim_\k \GrExt^i_A(A/A_{\geq n}, M)_\xi \\
  &\leq \sum_{p = -\infty}^\infty \dim_\k
  \GrExt^i_{\gr A}(\gr A / \gr (A_{\geq n}), M)_{(\xi, p)} \\
  &= \sum_{p = -\infty}^\infty \dim_\k H^i_{\gr \m}(\gr M)_{(\xi, p)}.
\end{align*}
\end{proof}

\paragraph
\label{transfer}
We are now ready to prove the main result of this section. 
\begin{Theorem*}
Suppose $A$ is a connected $\NN^r$-graded algebra endowed with a connected 
filtration, and that $\gr A$ is noetherian and has property $\chi$. Then the 
following hold.
\begin{enumerate}[label=(\alph*)]
\item $A$ has property $\chi$.
\item $\lcd_\m A \leq \lcd_{\gr \m} \gr A$.
\item If $\gr A$ is AS-Cohen-Macaulay, AS-Gorenstein or AS-Regular, so is $A$.
\end{enumerate}
\end{Theorem*}
\begin{proof}
The hypothesis that $\gr A$ is noetherian implies that $A$ is noetherian 
\cite{MR}*{1.6.9}. If $M$ is any finitely generated $\ZZ^r$-graded $A$-module, 
we may filter it using the procedure described in \ref{regular-filtration}. 
Since $\chi(\gr M)$ holds by hypothesis, item \ref{I:chi} of Corollary 
\ref{P:ext-ss} implies $\chi(M)$ holds, which proves that $A$ has property 
$\chi$. 

Item 2 follows from Corollary \ref{ldim}, as does the fact that if $\gr A$ is
AS-Cohen-Macaulay so is $A$. If $\gr A$ is AS-Gorenstein of injective 
dimension $n$ and Gorenstein shift $(\xi, p) \in \ZZ^{r+1}$, the algebra $A$ 
has injective dimension at most $n$ by item \ref{I:pd-id} of Corollary 
\ref{P:ext-ss}. Also, for each $\zeta \in \ZZ^r$ the spectral sequence 
$E(\k, A)_\zeta$ of Proposition \ref{P:ext-ss} degenerates at page $1$: it is 
zero if $\zeta \neq -\xi$, while for $\zeta = - \xi$ there is a single, one 
dimensional non-zero entry in the diagonal $-p-q = n$. Hence we obtain vector 
space isomorphisms
\begin{align*}
  \GrExt^i_A(\k, A)_\zeta 
    &\cong  \begin{cases}
      \k & \mbox{ if  $\zeta = -\xi$ and $i = n$;} \\
      0 & \mbox{otherwise.}
    \end{cases}
\end{align*}
Thus $\GrExt^i_A(\k,A) \cong \k[\xi]$ as $\ZZ^r$-graded vector spaces.
Any such isomorphism is also $A^\opp$-linear, so $A$ is left AS-Gorenstein of 
injective dimension $n$ and Gorenstein shift $\xi$. The same proof applies to 
show that $A$ is right AS-Gorenstein with the same injective dimension and 
Gorenstein shift, so $A$ is AS-Gorenstein. 

Finally, assume that $\gr A$ is AS-regular. Then $A$ is AS-Gorenstein, and item
\ref{I:pd-id} of Corollary \ref{P:ext-ss} implies that it has finite
left and right global dimensions. These dimensions are equal to the left and 
right projective dimensions of $\k$ as $A$-module \cite{RZ2}*{Lemma 2.1.5}, 
which in turn equal the left and right injective dimensions of $A$ and hence 
coincide, so $A$ is AS-regular.
\end{proof}

\section{Quantum affine toric degenerations}
\label{qatv}
In this section we recall the homological properties of quantum affine toric 
varieties proved in \cite{RZ2}, and give necessary and sufficient conditions 
for a connected $\NN^r$-graded algebra $A$ to have a connected filtration such 
that the associated graded ring is a quantum positive affine toric variety. 

\subsection*{Quantum affine toric varieties}

\paragraph
\label{affine-semigroups}
Recall that an affine semigroup is a finitely generated monoid isomorphic to a 
subsemigroup of $\ZZ^r$ for some $r \in \NN$. An embedding of $S$ is an 
injective semigroup morphism $i: S \hookrightarrow \ZZ^r$ for some $r \in 
\NN$. By definition, every affine semigroup $S$ is commutative and 
cancellative, so it has a group of fractions, which we denote by $G(S)$, and 
the natural map from $S$ to $G(S)$ is injective. The group $G(S)$ is a 
finitely generated and torsion-free commutative group, so there exists a 
natural number $r$ such that $G(S) \cong \ZZ^r$ as groups; we refer to $r$ as 
the \emph{rank} of $S$ and denote it $\rk S$. We say that $i$ is a \emph{full 
embedding} if the image of $S$ generates $\ZZ^r$ as a group, in which case 
$r = \rank S$. Fixing an isomorphism $G(S) \cong \ZZ^{\rank(S)}$ we obtain a 
full embedding of $S$ in an obvious way. 

\paragraph
\label{interior}
Let $S$ be an affine semigroup. Fixing an embedding $i: S \hookrightarrow 
\ZZ^r$ we identify $S$ with its image and see $\ZZ^r$ as a subgroup of $\RR^r$ 
in the obvious way. The \emph{rational cone} generated by $S$ in $\RR^r$ is 
the set 
\[
  \RR_+S 
    = \left\{\sum_{i=1}^n r_i s_i \mid r_i \in \RR_{\geq 0}, 
              s_i \in S, n \in \NN\right\}.
\]
The \emph{relative interior} of $S$ is defined as $\relint S = S \cap 
\RR_+S^\circ$, where $\RR_+S^\circ$ is the topological interior of $\RR_+ S$ 
as a subset of the vector space $\RR S$ endowed with the subspace topology 
induced from $\RR^r$. The relative interior is intrinsic to $S$ and does not 
depend on the chosen embedding \cite{BG}*{Remark 2.6}.

An affine semigroup $S$ of rank $r$ is called \emph{normal} if it verifies the 
following property: given $z \in G(S)$, if there exists $m \in \NN^*$ such 
that $mz \in S$, then $z \in S$. If we identify $S$ with a subset of $\ZZ^r$ 
through a full embedding and consider the real cone $\RR_+ S$ of $S$ inside 
$\RR^r$ in the obvious way, then Gordan's lemma \cite{BH}*{Proposition 6.1.2} 
states that $S$ is normal if and only if $S = \ZZ^r \cap \RR_+S$.

\begin{Example*}
We now introduce an example that will recur through the rest of this section.
Recall that a lattice is a poset $\L$ such that any two elements $x,y$ have
an infimum, called the \emph{meet} of $x$ and $y$ and denoted $x \meet y$,
and a supremum, called the \emph{join} of $x$ and $y$ and denoted $x \join y$.
For example, given $u \in \NN$ the poset $\NN^u$ with the product order is a 
lattice, with join (resp. meet) given by taking the maximum (resp. minimum) at 
each coordinate. A lattice is said to be \emph{finite} if its underlying set 
is finite, and \emph{distributive} if the binary operation $\meet$ is 
distributive over $\join$, and vice versa. Clearly $\NN^u$ is a distributive 
lattice, although of course not a finite one. Given $u,v \in \NN$ with $u 
\leq v$, the set $\Pi_{u,v}$ consisting of all $(a_1, \ldots a_u) \in \NN^u$ 
such that $1 \leq a_1 < \cdots <a_u \leq v$ is a finite distributive 
sublattice of $\NN^u$.

An element $z$ of a distributive lattice is said to be \emph{join-irreducible} 
if it is not the minimum, and whenever $z = x \join y$ then $z=x$ or $z=y$.  
Let $\L$ be a distributive lattice, and let $J(\L) = \{x_1, \ldots, x_r\}$ be 
the set of its join irreducible elements, enumerated so that $x_i < x_j$ 
implies $i<j$. Following \cite{RZ2}*{subsection 3.3} set $\str(\L)$ to be the 
subset of $\NN^{r+1}$ consisting of tuples $(a_0, a_1, \ldots, a_r)$ such that 
$a_0 \geq a_i$ for all $i = 1, \ldots, r$, and $a_i \geq a_j$ whenever $x_i < 
x_j$ as elements of $\L$. It turns out that $\str(\L)$ is a normal affine
semigroup, which we will call the affine semigroup associate to $\L$. Later on 
we will give an equivalent definition of $S(\L)$ in terms of the lattice, 
which will justify its name. Below we give an example with $\L = \Pi_{2,4}$.

\begin{tikzpicture}
\node (12) at (0.5,0) {$(1,2)$};
\node (13) at (1.5,1) {\fbox{$(1,3)$}};
\node (14) at (0,2) {\fbox{$(1,4)$}};
\node (23) at (3,2) {\fbox{$(2,3)$}};
\node (24) at (1.5,3) {$(2,4)$};
\node (34) at (3,4) {\fbox{$(3,4)$}};

\node (J) at (9,3.5) {$J(\Pi_{2,4}) = \{(1,3),(1,4),(2,3),(3,4)\}$};

\node (text) at (9,1.5) {$\str(\Pi_{2,4}) = \left\{(a,b,c,d,e) \in \NN^5 
 \  \Bigg\lvert 
  \begin{array}{rl}
  a &\geq b,c,d,e; \\
  b &\geq c,d,e; \\
  c,d &\geq e
  \end{array}\right\}$};

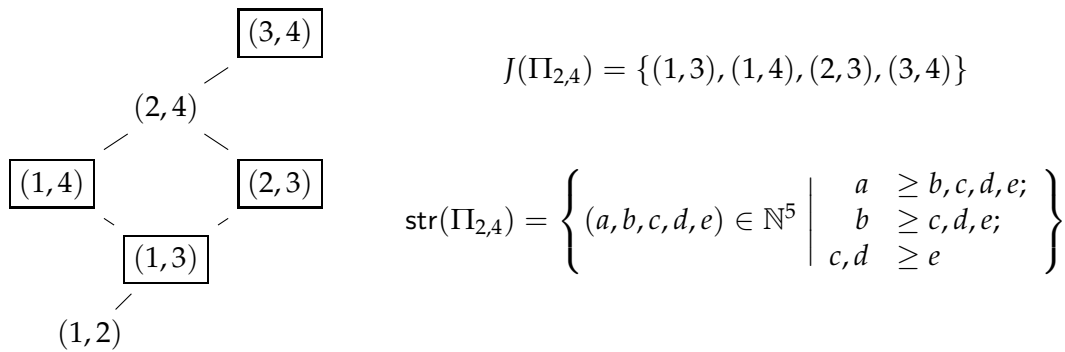
\node (caption) at (6,-1) {The lattice $\Pi_{2,4}$ and its affine semigroup. 
Framed elements are join-irreducibles.};
\draw (12) -- (13) -- (14) -- (24) -- (34) (13) -- (23) -- (24);
\end{tikzpicture}
\end{Example*}
\color{black}

\paragraph
\label{hilbert-basis}
Let $S$ be an affine semigroup. An element $s \in S$ is called 
\emph{irreducible} if whenever $s = x + y$ with $x,y \in S$, then either $x$ 
or $y$ is invertible. On the other hand $S$ is called \emph{positive} if its 
only invertible element is $0$. If $S$ is positive then the set of its 
irreducible elements is finite and generates $S$; for this reason it is called 
the \emph{Hilbert basis} of $S$. The fact that $S$ is positive also implies 
that there exists a full embedding $S \hookrightarrow \ZZ^r$ such that the 
image of $S$ is contained in $\NN^r$. It follows that a semigroup is positive 
if and only if there exists $r \geq 0$ such that $S$ is isomorphic to a 
finitely generated subsemigroup of $\NN^r$. Proofs of these results can be 
found in \cite{BG}*{pp. 54--56}. 

\paragraph
\label{presentation}
Let $S$ be a positive affine semigroup and let $\{s_1, \ldots, s_n\}$ be its 
Hilbert Basis. We denote by $\pi: \NN^n \to S$ the semigroup morphism defined 
by the assignation $e_i \mapsto s_i$ for each $1 \leq i \leq n$, where $e_i$ 
is the $i$-th element in the canonical basis of $\NN^n$. This map determines 
an equivalence relation in $\NN^n$ where $p \sim p'$ if and only if $\pi(p) = 
\pi(p')$; we denote this relation by $L(\pi)$. Clearly $L(\pi)$ is compatible 
with the additive structure of $\NN^n$, and hence the quotient $\NN^n / 
L(\pi)$ is a commutative monoid with the operation induced by addition in 
$\NN^n$, and there is an isomorphism of monoids $\NN^n/L(\pi) \cong S$. 

In general, an equivalence relation on $\NN^n$ closed under addition is called 
a \emph{congruence}. If $\rho$ is a subset of $\NN^n \times \NN^n$ then the 
congruence generated by $\rho$ is the smallest congruence containing the set 
$\rho$. Redei's theorem \cite{RGS}*{Theorem 5.12} states that every congruence 
in $\NN^n$ is finitely generated, i.e. there exists a finite set $\rho 
\subset \NN^n \times \NN^n$ that generates it; in particular there exists a 
finite set $P =  \{(p_1, p'_1), \ldots, (p_m,p'_m) \} \subset L(\pi)$ which 
generates $L(\pi)$. A \emph{presentation} of $S$ will be for us a pair 
$(\pi, P)$, where $\pi: \NN^n \to S$ is the map described above and $P$ is a 
finite generating set of the congruence $L(\pi)$. By the previous discussion 
every positive affine semigroup has a presentation.

\begin{Example*}
Let $\L$ be a finite distributive lattice, and let $F(\L)$ be the free 
commutative monoid on $\L$, which is clearly isomorphic to $\NN^{|\L|}$.
We will give a presentation of the group $\str(\L)$ introduced in 
\ref{interior} as a quotient of $F(\L)$.

As before, we denote by $\{x_1, \ldots, x_r\}$ the set of join-irreducible 
elements of $\L$. We denote by $[l]$ the image of an element $l \in \L$ in 
$F(\L)$ and define a map $\pi: F(\L) \to \str(\L)$ given by $\pi([l]) = e_0 + 
\sum_{x_i \leq l} e_i$. As shown in \cite{RZ2}*{Proposition 3.3.3} this map is 
surjective. Furthermore $L(\pi)$ is the equivalence relation generated by
the set $\{([l] + [l'], [l \join l'] + [l \meet l']) \mid l,l' \in \L\}$. 
Notice also that the Hilbert basis of $\str(\L)$ is precisely the image of 
$\L$, and the map $\pi$ restricted to $\L$ is a lattice morphism. Thus every 
distributive lattice $\L$ can be realized as a sublattice of 
$\{0,1\}^{|J(\L)|+1}$. 

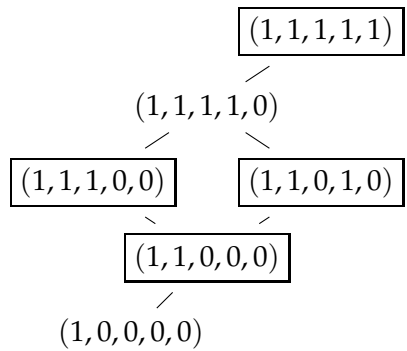
\begin{tikzpicture}
\node (12) at (0.5,0) {$(1,0,0,0,0)$};
\node (13) at (1.5,1) {\fbox{$(1,1,0,0,0)$}};
\node (14) at (0,2) {\fbox{$(1,1,1,0,0)$}};
\node (23) at (3,2) {\fbox{$(1,1,0,1,0)$}};
\node (24) at (1.5,3) {$(1,1,1,1,0)$};
\node (34) at (3,4) {\fbox{$(1,1,1,1,1)$}};

\node (caption) at (2,-1) {The lattice $\Pi_{2,4}$, with each element
replaced by its image through $\pi$.};

\draw (12) -- (13) -- (14) -- (24) -- (34) (13) -- (23) -- (24);
\end{tikzpicture}
\end{Example*}
\color{black}

\paragraph
\label{twisted-semigroup-algebras}
Let $S$ be a commutative semigroup with identity.
A $2$-cocycle over $S$ is a function $\alpha: S \times S \to \k^\times$ such 
that $\alpha (s,s') \alpha(s+s',s'') = \alpha(s,s'+s'')\alpha(s',s'')$ for all 
$s,s',s'' \in S$. Given a $2$-cocycle $\alpha$ over $S$, the 
\emph{$\alpha$-twisted semigroup algebra} $\k^\alpha[S]$ is the associative 
$\k$-algebra whose underlying vector space has basis $\{X^s \mid s \in S\}$ 
and whose product over these generators is given by $X^s X^t = \alpha(s,t)
X^{s+t}$. This is a noncommutative deformation of the classical semigroup 
algebra $\k[S]$. For more details see \cite{RZ2}*{section 3}; in this 
reference the group is assumed to be cancellative, but this is not relevant
for the definition of the twisted algebra.

\begin{Definition*}
Let $A$ be a connected $\NN^r$-graded algebra. We say that $A$ is a 
\emph{quantum positive affine toric variety} if there exist a positive affine 
semigroup $S$ and a $2$-cocycle $\alpha$ over $S$ such that $A$ is isomorphic 
to the twisted semigroup algebra $\k^\alpha[S]$, and for each $s \in S 
\setminus \{0\}$ the element $X^s$ is homogeneous of nonzero degree with 
respect to the $\NN^r$-grading of $\k^\alpha[S]$ induced by this isomorphism. 
In that case we refer to $S$ as the underlying semigroup of $A$.
\end{Definition*}

Let $r \geq 1$ and let $\psi: S \to \NN^r$ be a monoid morphism such that 
$\psi^{-1}(0) = \{0\}$. The twisted semigroup algebra $\k^\alpha[S]$ can be
endowed with a connected $\NN^r$-grading setting $\deg X^s = \psi(s)$ for each
$s \in S$. Conversely, any connected grading such that the elements of the 
form $X^s$ are homogeneous arises in this manner. In particular, if $A$ is a 
quantum positive toric variety with underlying semigroup $S$ then there is a 
corresponding monoid morphism $\psi: S \to \NN^r$, to which we will refer as 
the \emph{grading} morphism. 

\begin{Example*}
Let $q \in \k^\times$ and set $\mathbf{q} = \begin{pmatrix}
1 & q & 1 & 1 & q \\
q^{-1} & 1 & q & q & q^{-2} \\
1 & q^{-1} & 1 & 1 & q \\
1 & q^{-1} & 1 & 1 & q \\
q^{-1} & q^{2} & q^{-1} & q^{-1} & 1
\end{pmatrix}$. 

Let $\k_\q[X]$ be the twisted polynomial ring generated by variables $X_1, 
\ldots, X_5$ subject to the relations $X_i X_j = \q_{i,j} X_j X_i$. 
Given $t = (a,b,c,d,e) \in \NN^5$ we set $X^t = X_1^a X_2^b X_3^c X_4^d X_5^e$.
The formula $X^t X^s = \alpha(s,t) X^{t+s}$ induces a map $\alpha: \NN^5 \times
\NN^5 \to \k^\times$, and associativity of the product of $\k_\q[X]$ implies 
that $\alpha$ is a $2$-cocycle. Recall from the examples in \ref{interior}
and \ref{presentation} that $S = \str(\Pi_{2,4}) \subset \NN^5$ is an affine
semigroup. The restriction of $\alpha$ to $S$ is also a $2$-cocycle, and
$\k^\alpha[S]$ is isomorphic to the subalgebra of $\k_\q[X]$ generated by
the elements $X^t$ with $t$ running over the Hilbert basis of $S$. Setting
\begin{align*}
Y_1 &= X_1 & Y_2 &= X_1X_2 & Y_3 &= X_1 X_2 X_3 \\
Y_4 &= X_1 X_2 X_4 & Y_5 &= X_1 X_2 X_3 X_4 & Y_6 &= X_1 X_2 X_3 X_4 X_5
\end{align*} 
it is routine to check that $Y_j Y_i = q^{-2} Y_i Y_j$ for $(i,j) \in
\{(1,5), (2,4)\}$, that $Y_2$ and $Y_3$ commute, and that $Y_j Y_i
= q^{-1} Y_i Y_j$ for all other possible pairs $i < j$. There is one more
relation among these monomials, namely $Y_2Y_3 = q^{-1} Y_1 Y_4$. Thus 
$\k^\alpha[S]$ has as many generators as $S$, and can be presented by the 
commutation relations between them plus one extra relation arising from the 
fact that there is exactly one pair of incomparable elements in $\Pi_{2,4}$. 
We will see below in Proposition \ref{P:presentation} that all quantum 
positive affine toric varieties and several related algebras have similar 
presentations.
\end{Example*}

\paragraph
\label{properties-of-qatv}
Quantum positive affine toric varieties were studied from the point of view of 
noncommutative geometry in \cite{RZ2}*{section 3}. The following is a summary 
of the results proved there.
\begin{Proposition*}
Let $S$ be a positive affine semigroup and let $A$ be a positive quantum toric 
variety with underlying semigroup $S$. Then the following hold.
\begin{enumerate}[label=(\alph*)]
\item $A$ is noetherian and integral.

\item $A$ has property $\chi$ and finite local dimension equal to the rank of 
$S$.

\item Suppose $S$ is normal. Then $A$ is AS-Cohen-Macaulay and a maximal order 
in its division ring of fractions. Furthermore, $A$ is AS-Gorenstein if and 
only if there exists $s \in S$ such that $\relint S = s + S$.
\end{enumerate}
\end{Proposition*}
\begin{proof}
This is proved in \cite{RZ2}*{section 3.2} in the case where the grading 
morphism $\psi$ is given by a full embedding of $S$ in $\NN^r$ for some $r 
\geq 0$. The proposition follows from the fact that all the properties 
mentioned in it are independent of the grading, see paragraph 
\ref{invariance-by-grading}.
\end{proof}

\subsection*{Algebras with a quantum toric degeneration}
Classically, a toric degeneration of an algebraic variety $V$ is a flat 
deformation of $V$ into a toric variety $T$. Since varieties inherit many good 
properties from their flat deformations, and positive affine toric varieties 
are well studied, toric degeneration is a standard method to study algebraic 
varieties. With this in mind we introduce a noncommutative analogue of toric 
degeneration.  

\begin{Definition}
Let $A$ be a connected $\NN^r$-graded algebra. We will say that 
$A$ has a \emph{quantum positive affine toric degeneration} if it has a 
connected filtration such that its associated graded ring is a quantum 
positive affine toric variety. We refer to the underlying semigroup of this 
quantum positive affine toric variety as the semigroup associated to the 
degeneration.
\end{Definition}
For the sake of brevity we will write ``quantum toric degeneration'', omitting 
the adjectives ``positive'' and ``affine''. In view of Theorem \ref{transfer} 
and Proposition \ref{twisted-semigroup-algebras}, an algebra with a quantum 
toric degeneration is noetherian, integral, has property $\chi$ and finite 
local dimension. Furthermore, we can determine whether $A$ is 
AS-Cohen-Macaulay or AS-Gorenstein by studying the semigroup associated to the 
degeneration.

\paragraph
\label{s-phi-type}
A standard technique for proving toric degeneration of a variety is to find 
the structure of a Hodge algebra in the coordinate ring of the variety. We now 
introduce a noncommutative notion, inspired in the definition of Hodge 
algebras and its descendants such as classical and quantum algebras with a 
straightening law, which will play a similar role for the rest of this article.

\begin{Definition*}
Let $A$ be a noetherian connected $\NN^r$-graded algebra. Let $S$ be a 
positive affine semigroup and let $(\pi, P)$ be a presentation of $S$ as 
defined in \ref{presentation}, with $P = \{(p_i,p'_i) \mid 1 \leq i \leq m\}$. 
Let $\phi:~S~\to~\NN$ be a semigroup morphism such that $\phi^{-1}(0) = 
\{0\}$, and set $\tphi = \phi \circ \pi$.  

We say that the algebra $A$ is of $(S, \phi)$-type with respect to $(\pi, P)$
if the following hold. 
\begin{enumerate}
\item \label{generation}
$A$ is generated as algebra by a finite set of homogeneous elements $\{b_1, 
\ldots, b_n\}$ of the same cardinality as the Hilbert basis of $S$. We set
$\b^\xi = \displaystyle \prod_{i = 1}^n b_i^{\xi_i}$ for each $\xi \in \NN^n$.

\item \label{commutation}
For each $1 \leq i < j \leq n$ and each $\xi \in \NN^n$ such that $\tphi(\xi) 
< \phi(s_i + s_j)$, there exist $c_{i,j} \in \k^\times$ and $c^{i,j}_\xi \in 
\k$ such that
\begin{align*}
  b_j b_i = c_{i,j} b_i b_j + \sum_{\tphi(\xi) < \phi(s_i + s_j)} c^{i,j}_\xi
    \b^\xi.
\end{align*}

\item \label{straightening}
For each $1 \leq i \leq m$ and each $\xi \in \NN^n$ such that $\tphi(\xi) <
\tphi(p_i)$ there exist $d_i \in \k^\times$ and $d^i_\xi \in \k$ such that
    \begin{align*}
      \b^{p'_i} = d_i \b^{p_i} + \sum_{\tphi(\xi) < \tphi(p_i)} d^i_\xi \b^\xi.
    \end{align*}
\end{enumerate}
We say that $A$ is of $(S,\phi)$-type if there exists a presentation of $S$
such that $A$ is of $(S,\phi)$-type with respect to it.
\end{Definition*}

\begin{Remark*}
Since $\tphi(e_i) = \phi(s_i) > 0$ for all $i$, given $l \in \NN$ there exist 
finitely many $\xi \in \NN^n$ such that $\tphi(\xi) \leq l$ and so the sums on 
the right hand side of the formulas displayed in \ref{commutation} and 
\ref{straightening} are finite.
\end{Remark*}

The reader familiar with Hodge algebras will notice that there is a condition
missing in the definition, namely the existence of a set of linearly 
independent monomials on the generators. Thus the trivial algebra is of
$(S,\phi)$-type for any $S$ and $\phi$. This omission will be revised in 
\ref{P:equivalent-qatd}, where the existence of a linearly independent set
will be shown to be equivalent to the existence of a quantum toric
degeneration for the algebra.

\begin{Example*}
As stated above, this definition is inspired by that of quantum graded 
algebras with a straightening law (quantum graded ASL for short) as defined in 
\cite{LR1}*{Definition 1.1.1}. A quantum graded ASL is not necessarily an
algebra of $(S,\phi)$-type, but the results from \cite{RZ}*{section 5} show 
that the quantized coordinate rings of grassmannians and their Richardson 
subvarieties have this structure. Our paradigmatic example is the quantum
grassmannian.

Fix a field $\k$ and let $q \in \k^\times$. The quantum grassmannian 
$\mathcal G_q(2,4)$ is the algebra generated by elements $\{[ij], (i,j) \in 
\Pi_{2,4}\}$, subject to the commutation relations
\begin{align*}
[34][12]&= q^{-2}[12][34]
& [24][13] &= q^{-2}[13][24] + q^{-2}(q-q^{-1})[12][34]
\end{align*}
while $YX = q^{-1} XY$ for any other pair $X < Y$ (here we identify the 
generating set with $\Pi_{2,4}$ in the obvious way). There is also a quantum
Plücker relation, given by $[14][23] = q^{-1}[13][24] - q^{-2}[12][34]$.
Notice that the set of generators can be identified with $\Pi_{2,4}$, and 
inherits the structure of a distributive lattice.

Setting $\mathsf{wt}$ as in \cite{RZ}*{Definition 4.3} we obtain an 
assignation
\begin{align*}
[12] &\mapsto 21 & [13] &\mapsto 12 & [14] \mapsto 3 \\
[23] &\mapsto 11 & [24] &\mapsto 2  & [34] \mapsto 1
\end{align*}
which extends to a semigroup morphism $\phi: \str(\Pi_{2,4}) \to \NN$. It is
now a matter of routine computations to check that $\mathcal G_q(2,4)$ is of
$(\str(\Pi_{2,4}),\phi)$-type.
\end{Example*}
\color{black}

\paragraph
\label{type-filtration}
If $A$ is an algebra of $(S,\phi)$-type we write $F_lA = \langle \b^\xi \mid 
\tphi(\xi) \leq l \rangle$ for each $l \in \NN$. The fact that $\tphi(e_i) > 0$
for all $i$ implies that this is a finite dimensional vector space, and that 
$F_0A = \k$. The following lemma shows that $\F = \{F_lA\}_{l \geq 0}$ is a 
connected filtration on $A$.

\begin{Lemma*}
Let $S$ be a positive affine semigroup. Let $A$ be a noetherian connected 
$\NN^r$-graded algebra, and assume that it is of $(S, \phi)$-type with respect 
to a presentation $(\pi, P)$. Then the following hold.
\begin{enumerate}
\item \label{filtration}
Given $\xi, \nu \in \NN^n$ there exists $c_{\xi, \nu} \in \k^\times$ such that 
$\b^\xi \b^\nu \equiv c_{\xi, \nu}\b^{\xi + \nu} \mod F_{\tphi(\xi + \nu) - 1}
A$. In particular $\F$ is a filtration on $A$.

\item \label{straight}
Given $(\xi, \nu) \in L(\pi)$ there exists $d_{\xi, \nu} \in \k^\times$ such 
that $\b^\xi \equiv d_{\xi, \nu} \b^\nu \mod F_{\tphi(\xi) - 1} A$.
\end{enumerate}
\end{Lemma*}
\begin{proof}
To prove item \ref{filtration} we proceed by induction on $\tphi(\xi + \nu)$, 
with the $0$-th step being obvious since $F_0A = \k$ is a subalgebra of $A$. 
Suppose that the result holds for all $l < \tphi(\xi + \nu)$ and let $i$ be 
the least integer such that $\xi_i \neq 0$, so $\b^\xi \b^\nu = b_i (\b^{\xi - 
e_i} \b ^\nu)$. Using the inductive hypothesis we obtain
\begin{align*}
b_i (\b^{\xi - e_i} \b^\nu) 
&= b_i \bigg( c_{\xi- e_i, \nu} \b^{\xi - e_i + \nu} 
  + \sum_{\tphi(\rho) < \tphi(\xi - e_i + \nu)} c_\rho \b^\rho \bigg)\\
&= c_{\xi- e_i, \nu} b_i \b^{\xi - e_i + \nu} 
  + \sum_{\tphi(\rho) < \tphi(\xi - e_i + \nu)} c_\rho b_i \b^\rho
\end{align*}
where the $c_\rho \in \k$ and $c_{\xi - e_i,\nu} \in \k^\times$. The 
inductive  hypothesis also implies each product $b_i \b^\rho$ lies in 
$F_{\tphi(\xi + \nu) -1} A$, so
\[
  b_i \b^{\xi - e_i} \b^\nu 
    \equiv c_{\xi- e_i, \nu} b_i \b^{\xi - e_i + \nu} 
    \mod F_{\tphi(\xi + \nu)-1}A.
\]
Now let $j$ be the least integer such that $\nu_j \neq 0$. If $i \leq j$ then 
$b_i \b^{\xi - e_i + \nu} = \b^{\xi + \nu}$ and we are finished; otherwise, 
using item 2 of Definition \ref{s-phi-type} and a similar argument as before, 
we obtain
\begin{align*}
  b_i \b^{\xi - e_i + \nu} 
    = (b_i b_j)(\b^{\xi - e_i + \nu - e_j}) 
    \equiv c_{j,i} (b_j b_i)(\b^{\xi - e_i + \nu - e_j}) 
    \mod F_{\tphi(\xi + \nu) - 1}A
\end{align*}
The same reasoning applied to the product $b_j(b_i\b^{\xi - e_i + \nu - e_j})$ 
shows that
\begin{align*}
  b_j (b_i \b^{\xi - e_i + \nu - e_j}) 
    &\equiv c_{e_i, \xi - e_i + \nu - e_j} b_j(\b^{\xi  + \nu - e_j}) \mod
  F_{\tphi(\xi + \nu) - 1}A.
\end{align*}
Since $j < i$, the definition of $j$ implies $b_j(\b^{\xi  + \nu - e_j}) = 
\b^{\xi + \nu}$, so the proof of item \ref{filtration} is complete.

We now prove item \ref{straight}. Set
\[
  T = \{(\xi, \nu) \in L(\pi) \mid \b^\xi \equiv c \b^{\nu} \mod
    F_{\tphi(\xi) -1}
    \mbox{ for some } c \in \k^\times\}.
\]
We will show that $L(\pi) \subset T$, which clearly implies the desired result.
By definition $T$ is an equivalence relation, and item \ref{filtration} 
implies it is a congruence on $\NN^n$. By item \ref{straightening} of 
Definition \ref{s-phi-type}, every pair $(p_i, p'_i) \in P$ lies in $T$. Since 
$L(\pi)$ is the smallest congruence containing $P$, we deduce that $L(\pi) 
\subset T$.
\end{proof}

\paragraph
\label{P:equivalent-qatd}
Let $S$ be a positive affine semigroup and let $(\pi, P)$ be a presentation of 
$S$. A \emph{section} of $\pi$ is a function $\tau: S \to \NN^n$ such that 
$\pi \circ \tau = \Id_S$, that is $s = \sum_i \tau(s)_i s_i$ for every $s \in 
S$. If $A$ is an algebra of $(S,\phi)$-type then Lemma \ref{type-filtration} 
implies that for any section $\tau$ of $\pi$ the set $\{\b^{\tau(s)} \mid s 
\in S\}$ spans $A$. We now show that an algebra has a quantum toric 
degeneration with associated semigroup $S$ if and only if it is of 
$(S,\phi)$-type for an adequate morphism $\phi$ and the spanning set 
determined by any section is linearly independent.

\begin{Proposition*}
Let $S$ be a positive affine semigroup, and let $A$ be a noetherian connected 
$\NN^r$-graded algebra. The following statements are equivalent.
\begin{enumerate}
\item \label{has-qat-degeneration}
The algebra $A$ has a quantum toric degeneration with associated semigroup $S$.

\item \label{every-section}
For every presentation $(\pi, P)$ of $S$ there exists a semigroup morphism 
$\phi: S \to \NN$ such that $A$ is of $(S, \phi)$-type with respect to 
$(\pi, P)$, and for every section $\tau: S \to \NN^n$ of $\pi$ the set 
$\{\b^{\tau(s)} \mid s \in S\}$ is linearly independent.

\item \label{some-section}
There exist a presentation $(\pi, P)$, a semigroup morphism $\phi: S \to \NN$ 
and a section $\tau: S \to \NN^n$ of $\pi$ such that $A$ is of 
$(S, \phi)$-type with respect to $(\pi, P)$ and the set $\{\b^{\tau(s)} \mid s 
\in S\}$ is linearly independent.
\end{enumerate}
\end{Proposition*}
\begin{proof}
We first show that \ref{has-qat-degeneration} implies \ref{every-section}. By 
hypohtesis there exists a filtration by graded subspaces $\F = \{F_lA\}_{l 
\geq 0}$ such that $\gr_\F A \cong \k^\alpha[S]$ as $\NN^{r+1}$-graded 
algebras for some $2$-cocycle $\alpha: S \times S \to \k^\times$, with the 
grading on the twisted semigroup algebra given by a semigroup morphism 
$\psi: S \to \NN^{r+1}$ such that $\psi^{-1}(0) = \{0\}$. We identify 
$\gr_\F A$ with $\k^\alpha[S]$ through this isomorphism to simplify notation. 

Fix a presentation $(\pi, P)$ of $S$. For each $1 \leq i \leq n$ we choose 
homogeneous elements $b_i \in A$ such that $\gr b_i = X^{s_i}$. By definition 
of the product of an associated graded ring, for each $\xi \in \NN^n$ the 
element $\prod_{i = 1}^n (\gr b_i)^{\xi_i}$ equals either $\gr 
\left(\prod_{i = 1}^n b_i^{\xi_i}\right)$ or zero. Since $\k^\alpha[S]$ is an 
integral ring the last possibility cannot occur, so $\gr (\b^\xi)$ equals a 
nonzero multiple of $X^{\pi(\xi)}$. Thus if $\tau: S \to \NN^n$ is a section 
of $\pi$ then for each $s \in S$ there exists a nonzero constant $c_s$ such 
that $\gr b^{\tau(s)} \equiv c_s X^s$, and so the set $\{\gr \b^{\tau(s)} 
\mid s \in S\}$ is a basis of $\k^\alpha[S]$, which implies that 
$\{\b^{\tau(s)} \mid s \in S\}$ is a basis of $A$. This also proves that $A$ 
satisfies item \ref{generation} of Definition \ref{s-phi-type}.

Let $\phi: S \to \NN$ be the additive map given by $s \mapsto \psi(s)_{r+1}$; 
equivalently $\phi(s)$ is the minimal $l$ such that $\b^\xi \in F_l A$ for all 
$\xi \in \pi^{-1}(s)$. In particular $\phi(s_i) > 0$ since $F_0A = \k$. Also 
$F_lA = \vspan{\b^\xi \mid \tphi(\xi) \leq l}$, and since for each $\xi \in 
\NN^n$ there exists a nonzero constant $c_\xi$ such that $\gr b^\xi \equiv 
c_\xi \gr b^{\tau(\pi(\xi))}$ we actually have $F_lA = \vspan{\b^{\tau(s)} 
\mid \phi(s) \leq l}$. Finally, for each $1 \leq i,j \leq n$ and each $1 
\leq k \leq m$ there exist $c_{i,j}, d_k \in \k^\times$ such that
\begin{align*}
  \gr (b_j b_i) &= c_{i,j} \gr (b_i b_j),\\
  \gr \b^{p'_k} &= d_k \gr \b^{p_k}
\end{align*}
hold in $\gr_\F A$, which implies that items \ref{commutation} and 
\ref{straightening} of Definition \ref{s-phi-type} hold in $A$ for the morphism
$\phi$ we have just defined. Thus $A$ is of $(S,\phi)$-type, and we have 
proved \ref{has-qat-degeneration} implies \ref{every-section}.

We said in \ref{presentation} that every positive affine semigroup has a 
presentation so clearly \ref{every-section} implies \ref{some-section}. Let us 
see that \ref{some-section} implies \ref{has-qat-degeneration}. Define the 
filtration $\F = \{F_l A\}_{l \geq 0}$ as in \ref{type-filtration}. By item 2 
of Lemma \ref{type-filtration}, the set $\{\b^{\tau(s)} \mid \phi(s) \leq l\}$ 
generates $F_lA$ for each $l \in \NN$, and since by hypothesis it is linearly 
independent, it is a basis of $F_lA$. Hence $\gr_\F A$ is generated by $\{\gr 
\b^{\tau(s)} \mid s \in S\}$. Once again by Lemma \ref{type-filtration} for 
each $s,s' \in S$ there exist $\beta(s,s'), \alpha(s,s') \in \k^\times$ such 
that
\begin{align*}
(\gr \b^{\tau(s)})(\gr \b^{\tau(s')}) 
  = \beta(s,s') \gr \b^{\tau(s) + \tau(s')} 
  = \alpha(s,s') \gr \b^{\tau(s+s')}.
\end{align*}
Associativity of the product of $\gr_\F A$ implies that $\alpha: S \times S \to
\k^\times$ is a $2$-coycle, so we may consider the $\k$-linear map 
$\k^\alpha[S] \to \gr_\F A$ induced by the assignation $X^s \mapsto \gr 
\b^{\tau(s)}$, which is a multiplicative map. Since $S$ is positive we must 
have $\tau(0) = 0$, and hence $\alpha(s,0) = \alpha(0,s) = 1$ for all $s \in S$
which implies that our multiplicative map is unitaty and hence an isomorphism 
of $\k$-algebras. Furthermore, the elements $\gr \b^{\tau(s)}$ are
homogeneous, so this algebra is indeed a quantum positive affine toric variety.
\end{proof}

\begin{Example*}
We return one last time to the example of the quantum grassmannian 
$\mathcal G_q(2,4)$ over an arbitrary field discussed in \ref{s-phi-type}.
We extend the order of $\Pi_{2,4}$ to a total order so we can identify 
the free abelian semigroup over $\Pi_{2,4}$ with $\NN^6$, and let $\pi: \NN^6
\to \str(\Pi_{2,4})$ be the presentation morphism described in 
\ref{presentation}. For each $s \in \str(\Pi_{2,4})$ the fiber $\pi^{-1}(s)$
is finite and has a unique maximal element with respect to the total 
lexicographic order of $\NN^6$, which we denote by $\tau(s)$. The monomials 
corresponding to this section are precisely the standard monomials introduced 
in \cite{LR1}*{Definition 3.2.1}. Using the map $\phi$ from \ref{s-phi-type} 
the quantum toric variety obtained by degeneration is the twisted semigroup 
algebra presented in \ref{twisted-semigroup-algebras} (this explains the 
rather odd choice of commutation coefficients in that example).

A similar argument holds not just for quantum grassmannians, but for the large 
class of quantum graded ASL satisfying condition $(C)$ introduced in 
\cite{RZ}*{Definition 4.1}. This includes all quantum grassmannians in type
$A$, along with their Schubert and Richardson subvariaties. In the following
section we will show that quantized coordinate rings of Schubert subvarieties 
of arbitrary flag varieties also have a quantum toric degeneration.
\end{Example*}
\color{black}

\paragraph
\label{lex-degeneration}
We now introduce a second notion related to quantum toric degenerations. 
Recall that a commutative semigroup $S$ is said to be well-ordered if there 
exists a well-order $<$ on $S$ compatible with the additive structure, i.e. 
such that for all $s,s',s'' \in S$ the inequality $s<s'$ implies $s+s'' < 
s'+s''$. 

\begin{Definition*}
Let $S$ be a commutative semigroup, and let $<$ be a well-order on $S$ 
compatible with the semigroup structure. Let $A$ be a connected $\NN^l$-graded 
algebra for some $l \in \NN$. An \emph{$(S,<)$-basis} for $A$ is an ordered 
basis $\{b_s \mid s \in S\}$ consisting of homogeneous elements, such that for 
all $s,s',s'' \in S$ with $s'' < s+s'$ there exist $c_{s,s'} \in \k^\times$ 
and $c_{s,s'}^{s''} \in \k$ such that
\[
  b_s b_{s'} = c_{s,s'} b_{s+s'} + \sum_{s'' < s+s'} c_{s,s'}^{s''} b_{s''}.
\]
\end{Definition*}

\paragraph
\label{S-order-bases}
Let $S$ be a commutative semigroup with neural element $0$. Assume $<$ is a 
well-order on $S$ compatible with its additive structure, and let $A$ be an 
algebra. An $(S,<)$-filtration on $A$ is a collection of vector spaces $\F = 
\{F_sA \mid s \in S\}$, such that $F_sA \cdot F_{s'} A \subset F_{s+s'}A$ for 
all $s,s' \in S$, and such that $F_sA \subseteq F_{s'}A$ whenever $s \leq s'$. 
The standard notions related to $\NN$-filtrations translate easily to the 
context of $(S,<)$-filtrations. We will assume that our filtrations are always 
\emph{exhaustive}, so $A = \bigcup_{s \in S} F_s A$, and \emph{discrete},
so $F_{s}A = 0$ for all $s<0$.

Let $A$ be an $S$-filtered algebra. The associated graded algebra 
$\gr_\F A$ is defined setting $F_{<s}A = \sum_{t<s} F_tA$ and taking
\[
  \gr_\F A = \bigoplus_{s \in S} \frac{F_sA}{F_{<s}A}.
\] 
As usual, for each element $a \in A$ we may define $\gr a$ as the image of $a$ 
in the quotient $F_sA / F_{<s}A$ where $s$ is the first element of $S$ such 
that $a \in F_sA$; notice that this element exists because $<$ is a 
well-order. The product can then be defined as in the $\NN$-filtered case, 
namely if $a,b \in A$ and $s,t$ are minimal elements such that $a \in F_sA$ 
and $b \in F_tA$, then $(\gr a)(\gr b)$ equals the image of $ab$ in $F_{s+t}A 
/ F_{<s+t}A$, which equals $\gr(ab)$ if $t+s$ is minimal with respect to the 
property that $ab \in F_{t+s}A$ and zero otherwise.

\begin{Lemma*}
Let $S$ be a positive affine semigroup and let $<$ be a well-order of $S$
compatible with the semigroup structure. 
Let $A$ be a noetherian $\NN^l$-graded connected algebra for some $l \in \NN$, 
and assume $A$ has an $(S,<)$-basis $\B = \{b_s \mid s \in S\}$. 
Set $F_sA = \langle b_t \mid t \leq s \in S \rangle$ and $\F = \{F_sA \mid s 
\in S\}$. Also let $\{s_1, \ldots, s_n\}$ be the Hilbert basis of $S$ and set 
$b_i = b_{s_i}$ for all $1 \leq i \leq n$. The following hold.
\begin{enumerate}
\item The family $\F$ is an $(S,<)$-filtration. Furthermore,
each quotient $F_sA / F_{<s}A$ is of dimension $1$ and $F_0A = \langle 1 
\rangle$.

\item There exists a $2$-cocycle $\alpha$ over $S$ such that $\gr_\F A$ is 
isomorphic as $S$-graded algebra to $\k^\alpha[S]$ for the obvious $S$-grading
on this.

\item The algebra $A$ is generated by the set $\{b_i \mid 1 \leq i \leq n\}$.
\end{enumerate} 
\end{Lemma*}
\begin{proof}
The fact that $\F$ is an exhaustive and discrete $(S,<)$-filtration is an 
immediate 
consequence of the definition of an $(S,<)$-basis. Also $F_{<s}A = \langle b_t 
\mid t < s\rangle$, so $F_sA / F_{<s}A$ is generated by the image of $b_s$ in 
the quotient, which is nonzero. Finally, writing $1$ as a linear combination 
of the $b_s$ and using a leading term argument, it is easy to see that $1 \in 
F_0A$ and hence it must generate it. Notice that this implies that $b_0$ is a 
scalar so without loss of generality we may assume that $b_0 = 1$. This proves
item $1$.

Set $\alpha(s,s') = c_{s,s'}$ for each $s,s' \in S$. By the first item the set 
$\{\gr b_s \mid s \in S\}$ is a basis of $\gr_\F A$, and $(\gr b_s) 
(\gr b_{s'}) = c_{s,s'} \gr b_{s+s'}$ by definition of the product on the 
associated graded ring. Associativity of the product in $\gr_\F A$ implies 
then that 
$\alpha$ is a $2$-cocycle over $S$, and furthermore the map $\gr_\F A \to 
\k^\alpha[S]$ sending $\gr b_s$ to $X^s$ is a multiplicative $S$-graded 
vector-space isomorphism. Since we are assuming that $b_0 = 1$, it follows 
that $\alpha(1,s) = \alpha(s,1) = 1$ for all $s \in S$ and hence our 
isomorphism preserves the unit, and is thus a ring isomorphism.
This proves item 2.

Finally, in order to prove that the $b_i$'s generate $A$ it is enough to show 
that each $b_s$ is in the algebra generated by these elements. Suppose this is 
not the case. Then, since $S$ is well-ordered by $<$, there exists a minimal 
$s$ such that $b_s$ is not in the algebra generated by the $b_i$'s. Take $\xi 
= (\xi_1, \ldots, \xi_n) \in \NN^n$ such that $\sum_i \xi_i s_i = s$. By the 
definition of the product in the associated graded ring, $\prod_i
(\gr b_i)^{\xi_i}$ equals either $\gr (\b^\xi)$ or zero, and since $\gr_\F A 
\cong \k^\alpha[S]$ is integral (see \cite{RZ2}*{Lemma 3.2.3}) the second possibility can not occur. Thus 
$\gr(\b^\xi)$ is a nonzero element of $\gr_\F A$ of degree $\pi(\xi)$, so item 
1 of this lemma implies that $\gr b_s = c\gr(\b^\xi)$ for some $c \in 
\k^\times$, and hence $b_s  = c\b^\xi + \sum_{t<s} c_t b_t$. By the minimality 
of $s$ all the $b_t$'s appearing in the sum on the right hand side of the 
equation lie in the algebra generated by the $b_i$'s, and clearly so does 
$\b^\xi$, a contradiction. 
\end{proof}

\paragraph
\label{S-ordered-basis-degeneration}
There is an obvious way to obtain well-orderings on positive affine 
semigroups. Since $S$ is a positive affine semigroup it can be embedded in 
$\NN^r$ for some $r \geq 0$ through a monoid morphism $\iota: S \to \NN^r$. 
Now $\NN^r$ is a well-ordered semigroup with the lexicographic order, which is 
compatible with its additive structure, so we may pull-back the lexicographic 
order through $\iota$ and thus obtain a well-order $<^\iota$ over $S$, which 
is also compatible with its additive structure. Notice that in this case $0$ 
is always the minimal element of $S$.

\begin{Proposition*}
Let $S$ be a positive affine semigroup, and let $A$ be a noetherian connected 
$\NN^l$-graded algebra for some $l \in \NN$. The algebra $A$ has a quantum 
affine toric degeneration with underlying semigroup $S$ if and only if there 
exists an embedding $\iota: S \to \NN^t$ such that $A$ has an ordered 
$(S, <^\iota)$-basis.
\end{Proposition*}
\begin{proof}
Suppose $A$ has a quantum affine toric degeneration with underlying semigroup 
$S$. Then by Proposition \ref{P:equivalent-qatd} there exists a semigroup 
morphism $\phi: S\to \NN$ such that $A$ is of $(S, \phi)$-type, and we may 
choose any section $\tau: S \to \NN^n$ to obtain a basis $\mathcal B = 
\{\b^{\tau(s)} \mid s \in S\}$. Let $\rho: S \to \NN^{t-1}$ be an embedding 
($t>1$) and let $\iota: S \to \NN^t$ be defined as $\iota(s) = (\phi(s), 
\rho(s))$, which is an embedding of $S$ since $\rho$ is an embedding. Write 
$<$ for $<^\iota$, and notice that $\phi(s) < \phi(s')$ implies $s<s'$. By 
Lemma \ref{type-filtration}, for all $s,s' \in S$ and all $s''$ such that 
$\phi(s'') < \phi(s+s')$ there exist $c_{s,s'} \in \k^\times$ and 
$c_{s,s'}^{s''} \in \k$ such that
\[
  \b^{\tau(s)}\b^{\tau(s')} 
    = c_{s,s'}\b^{\tau(s+s')} 
      + \sum_{\phi(s'') < \phi(s+s')} c_{s,s'}^{s''} \b^{\tau(s'')},
\]
which implies $\mathcal B$ is an ordered $S$-basis with respect to $<$.

Now assume $A$ has an ordered $S$-basis with respect to some total order $<$ 
induced by an embedding $\iota: S \to \NN^t$. Since $<$ is the pull-back of 
the lexicographic order through an embedding, we might as well assume $S 
\subset \NN^t$ and that $<$ is the lexicographic order. By the previous lemma, 
we already know that the $b_i$'s generate $A$, so all that is left to do is to 
prove the existence of an additive map $\phi:S \to \NN$ and that the 
desired relations hold.

As before, we denote by $\pi: \NN^n \to S$ the map $(\xi_1, \ldots, \xi_n)
\mapsto \sum_i \xi_i s_i$.
Recall that using the $(S,<)$-filtration $\F$ defined in the previous lemma, 
we proved that $\gr_\F A \cong \k^\alpha[S]$ for some $2$-cocycle $S$. This 
implies that for all $1 \leq i < j \leq n$ and all $\xi \in \NN^n$ such that 
$\pi(\xi) < s_i +s_j$ there exist $c_{i,j} \in \k^\times$ and $c^{i,j}_\xi 
\in \k$ such that
\begin{align*}
b_j b_i 
  = c_{i,j} b_i b_j 
    + \sum_{\pi(\xi) < \pi(s_i + s_j)} c^{i,j}_\xi \b^\xi,
\end{align*}
and for each $1 \leq i \leq m$ and each $\xi \in \NN^n$ such that $\pi(\xi) < 
p_i$ there exist $d_i \in \k^\times$ and $d^i_\xi \in \k$ such that
\begin{align*}
  \b^{p'_i} = d_i \b^{p_i} + \sum_{\pi(\xi) < p_i} d^i_\xi \b^\xi.
\end{align*}
Let $C \subset S$ be the set consisting of the following elements:
\begin{itemize}
\item all $s_i + s_j$ with $1 \leq i<j\leq n$;
\item all $\pi(p_i)$ with $1 \leq i \leq m$;
\item all $\pi(\xi)$ such that $c^{i,j}_\xi \neq 0$ for some $1 \leq i < j 
  \leq n$;
\item and all $\pi(\xi)$ such that $d^i_\xi \neq 0$ for some $1 \leq i \leq m$.
\end{itemize}  
The set $C$ is finite and hence is contained in a cube $[0,N]^t$ for $N$ 
large enough. Set $\phi: \NN^t \to \NN$ to be the semigroup morphism defined 
by $(c_1, \ldots, c_t) \mapsto \sum_{i=1}^t c_i (N+1)^{t-i}$. If $c \in C$ 
then $\phi(c)$ is the unique natural number such that its $N+1$-adic expansion
has $c_i$ as its $i$-th digit. This implies that $\phi$ respects the 
restriction of the lexicographic order to $C$, and thus $A$ is of 
$(S,\phi)$-type. 

Let $\tau: S \to \NN^n$ be any section of $\pi$. The algebra $\gr_\F A$ has a 
natural $S$-grading, and for each $s \in S$ the element $\gr \b^{\tau(s)}$ is 
of degree $s$. As we have already observed, this is a non-zero element so the 
set $\{\gr \b^{\tau(s)} \mid s \in S\}$ is a basis of $\gr_\F A$, which 
implies that $\{\b^{\tau(s)} \mid s \in S\}$ is a basis of $A$. Thus by 
Proposition \ref{P:equivalent-qatd} $A$ has a quantum affine toric 
degeneration with underlying semigroup $S$.
\end{proof}

\begin{Remark*}
The trick of turning the $S$-filtration into an $\NN$-filtration using 
$N+1$-adic expansions is due to Caldero \cite{C}*{Lemma 3.2}. A similar though 
less general version of this idea appears in \cite{GL} and \cite{RZ}. 
\end{Remark*}

\paragraph
\label{P:presentation}
We finish this section with an easy consequence of Lemma 
\ref{type-filtration}. It will not be used in the sequel, but we include it 
for completeness.
\begin{Proposition*}
Let $S$ be a positive affine semigroup and $A$ a noetherian connected $\NN^r$-
graded algebra. If $A$ is of $(S,\phi)$-type for some monoid morphism $\phi: 
S \to \NN$, and there exists some section $\tau$ of $\pi$ such that the set 
$\{\b^{\tau(s)} \mid s \in S\}$ is linearly independent, then the relations 
given in items \ref{commutation} and \ref{straightening} of Definition 
\ref{s-phi-type} give a presentation of $A$.
\end{Proposition*}
\begin{proof}
Since $A$ is of $(S,\phi)$-type, it is generated as algebra by homogeneous 
elements $b_1, \ldots, b_n$, and there exist constants $c_{i,j}, c^{i,j}_\xi, 
d_i, d^i_\xi$ such that $A$ complies with Definition \ref{s-phi-type}. 
Furthermore, the relations described in items \ref{commutation} and 
\ref{straightening} of this definition are homogeneous.

Let $B$ be the free algebra generated by $X_1, \ldots, X_n$ and let $I$ be the 
ideal of $B$ generated by the elements
\begin{align*}
  X_j X_i &- c_{i,j} X_i X_j - \sum_{\tphi(\xi) < \phi(s_i + s_j)} c^{i,j}_\xi
  X^\xi, 
    & \mbox{for } 1 \leq i < j \leq n;\\
  X^{p'_i} &- d_i X^{p_i} - \sum_{\tphi(\xi) < \tphi(p_i)} d^i_\xi X^\xi, 
    &\mbox{for } 1 \leq i \leq m;
\end{align*}
where $X^\xi = X_1^{\xi_1} X_2^{\xi_2} \cdots X_n^{\xi_n}$ for each $\xi \in 
\NN^n$. We put an $\NN^r$ grading on $B$ by setting $\deg X_i = \deg b_i$, and 
this induces an $\NN^r$ grading on $B/I$. Since $B$ is a free algebra the 
assignation $X_i \mapsto b_i$ induces a morphism of graded algebras $B \to A$, 
which factors through $B/I$. We thus obtain a morphism of $\NN^r$-graded 
algebras $f: B/I \to A$.

We denote by $Y_i$ the image of $X_i$ in $B/I$. Clearly $B/I$ is an $(S,\phi)$-
algebra, and the algebra map $f: B/I \to A$ sends $Y_i$ to $b_i$ for all $i$. 
Since $f(Y^{\tau(s)}) = \b^{\tau(s)}$, the set $\{Y^{\tau(s)} \mid s \in S\}$ 
is linearly independent and hence a basis of $B/I$. Thus $f$ maps a basis onto 
a basis, so it is an isomorphism.
\end{proof}

\section{Quantum affine toric degeneration of quantum Schubert varieties}
\label{q-flag-schubert}
We apply the results in the previous section to study Schubert varieties 
of quantum flag varieties. We recall the definitions of quantum flag and 
Schubert varieties with some detail in order to establish notation. We then
adapt an argument due to P. Caldero to show that these algebras have 
$(S,<)$-bases for adequate semigroups $S$. The main ingredient in the 
construction of these bases is the canonical or global basis of $U_q^-(\g)$ 
discovered independently by Lusztig and Kashiwara. The semigroup arises out
of the string parametrization of this basis.

\subsection*{Quantum flag and Schubert varieties}

\paragraph
\label{lie-algebra-notation}
Let $\g$ be a complex semisimple Lie algebra. We denote by $\Phi$ the root
system of $\g$ with respect to a fixed Cartan subalgebra, and by $\ZZ\Phi$ its 
root lattice. We also fix a basis $\Pi \subset \Phi$ of positive roots, and 
write $P = \{\varpi_\alpha \mid \alpha \in \Pi\}$ for the set of 
corresponding fundamental weights. We denote by $\Lambda$ the weight lattice 
$\sum_{\alpha \in \Pi} \ZZ \varpi_\alpha$ and by $\Lambda^+$ the set of 
dominant integral weights $\sum_{\alpha \in \Pi} \NN \varpi_\alpha$. 

Let $W$ be the Weyl group of $\g$, and $s_\alpha \in W$ the reflection 
corresponding to $\alpha \in \Pi$. Given an element $w \in W$ we denote 
its length by $\ell(w)$, and set $N$ to be the length of $w_0$, the longest 
element of $W$. A \emph{decomposition} of $w \in W$ is a word on the 
generators $s_\alpha$ that equals $w$ in $W$. The decomposition is 
\emph{reduced} if it is of minimal length, i.e. its length equals $\ell(w)$. 
We denote by $(-,-)$ the standard $W$-invariant pairing between $\ZZ\Phi$ and 
$\Lambda$, and write $\langle \lambda, \alpha\rangle = 
\frac{2(\lambda,\alpha)}{(\alpha,\alpha)}$ for all $\lambda \in \Lambda$ and
$\alpha \in \Phi$, so if $\lambda = \sum_{\alpha \in \Pi} r_\alpha 
\varpi_\alpha$ then $\langle \lambda, \alpha \rangle = r_\alpha$.

\paragraph
\label{quantized-enveloping-algebra}
Fix $q \in \k^\times$. Let $U_q(\g)$ be the quantum enveloping algebra of 
$\g$; this is an algebra generated by elements $F_\alpha, E_\alpha, 
K_\alpha^{\pm 1}$
for $\alpha \in \Pi$, with the relations given in \cite{Jan}*{Definition 4.3}. 
We denote by $U_q^+(\g), U_q^-(\g)$ the subalgebras generated by the 
$E_\alpha$'s and the $F_\alpha$'s, which are respectively called the positive 
and negative parts of $U_q(\g)$ \cite{Jan}*{4.4}. As shown in 
\cite{Jan}*{Proposition 4.11}, $U_q(\g)$ is a Hopf algebra.

If $q$ is not a root of unity and $\mathsf{char} \k > 3$ then by \cite{Jan}*{chapter 5} 
for each $\lambda 
\in \Lambda^+$ there is an irreducible highest-weight representation of 
$U_q(\g)$ of type $\mathbf{1}$, which we denote by $V_q(\lambda)$. Each 
$V_q(\lambda)$ decomposes as the direct sum of weight spaces $\bigoplus_{\mu 
\in \Lambda} V_q(\lambda)_\mu$; the dimensions of the weight spaces are the 
same as the corresponding representation over $\g$, so the Weyl character 
formula holds for these representations, see \cite{Jan}*{5.15}. 

\paragraph
Let $G$ be the simply connected, connected algebraic group with Lie algebra 
$\g$. Since $U_q(\g)$ is a Hopf algebra, its dual $U_q(\g)^*$ is an algebra 
with convolution product induced by the coproduct of $U_q(\g)$. There is a map 
$V_q(\lambda)^* \ot V_q(\lambda) \to U_q(\g)^*$ defined by
sending $\phi \ot v \in V_q(\lambda)^* \ot V_q(\lambda)$ to the linear 
functional $c^\lambda_{\phi, v}$, which assigns to each $u \in U_q(\g)$ the 
scalar $c^\lambda_{\phi, v}(u) = \phi(uv)$. Functionals of the form 
$c^\lambda_{\phi, v}$ are called \emph{matrix coefficients}. The $\k$-linear 
span of the matrix coefficients is a subalgebra of $U_q(\g)^*$ denoted by 
$\O_q[G]$, called the quantized algebra of coordinate functions over the group 
$G$ \cite{Jan}*{7.11}.

\paragraph
\label{q-full-flag-varieties}
Quantum analogues of flag varieties and their Schubert subvarieties were 
introduced by Soibelman in \cite{S} and by Lakshmibai and Reshetikhin in 
\cite{qLR}; we review their definition. 
We assume that $q$ is not a root of unity. Fix a maximal Borel subgroup $B$ of
$G$. The \emph{full flag variety} associated to $G$ is $G/B$. Let 
$C^+_q(\lambda)$ be the vector space of matrix coefficients of the form 
$c^\lambda_{\phi,v_{\lambda}}$ in $U_q(\g)^*$, where $v_\lambda$ is a highest 
weight vector in $V_q(\lambda)$, and set 
\[ 
  \O_q[G/B] = \bigoplus_{\lambda \in \Lambda^+} C^+_q(\lambda) \subset \O_q[G] 
    \subset U_q(\g)^*.  
\] 
This is called the \emph{quantum full flag variety} of $G$. The product of two 
matrix coefficients in $\O_q[G/B]$ is again in $\O_q[G/B]$, and its 
decomposition as a direct sum gives $\O_q[G/B]$ the structure of a 
$\Lambda^+$-graded algebra. 

Let $I$ be a subset of the set of fundamental weights and set $\mathcal J(I) = 
\sum_{\varpi \notin I} \NN \varpi$. Denote by $W_I \subset W$ the subgroup 
generated by the reflections $s_\alpha$ with $\varpi_\alpha \in I$, and for 
each class in $W/W_I$ pick a representative of minimal length. We denote by 
$W^I$ the set of these representatives. Since the Weyl character formula 
holds, for each $w \in W$ and each $\lambda \in \mathcal J(I)$ the vector 
space $V_q(\lambda)_{w\lambda}$ has dimension $1$. The \emph{Demazure module} 
$V_q(\lambda)_{w}$ is the $U_q^+(\g)$-submodule of $V_q(\lambda)$ generated by 
a vector of weight $w\lambda$ in $V_q(\lambda)$.

The set $I$ determines a Lie subalgebra $\mathfrak p \subset \g$, and a 
parabolic subgroup $P_I \subset G$. The variety $G/P_I$ is the corresponding 
generalized flag variety. To these data we associate the $\Lambda^+$-graded 
subalgebra of $\O_q[G/B]$ 
\[ 
  \O_q[G/P_I] = \bigoplus_{\lambda \in \mathcal J(I)} C_q^+(\lambda) 
\] 
called the \emph{quantum partial flag variety associated to $I$}. 

Given vector spaces $V_2 \subset V_1$, we denote by $V_2^\perp$ the set of 
linear functionals over $V_1$ which are zero on $V_2$. For every $w \in W^I$ 
the vector space 
\[
  J_w^I = \bigoplus_{\lambda \in \mathcal J(I)} \vspan{c_{\phi,v_\lambda}^\lambda
  \in C_q^+(\lambda) \mid \phi  \in V_q(\lambda)_w^\perp } \subset \O_q[G/P_I] 
\] 
is an ideal of $\O_q[G/P_I]$ called the \emph{Schubert ideal} associated to 
$w$. The quotient algebra $\O_q[G/P_I]_w = \O_q[G/P_I]/J^I_w$ is called the 
\emph{quantum Schubert variety} associated to $w$.

\subsection*{Degeneration of quantum Schubert varieties}
Our aim is to show that quantum Schubert varieties have quantum affine toric 
degenerations. In order to do so we work for a moment over the field $\QQ(v)$, 
where $v$ is an indeterminate over $\QQ$, and consider the $\QQ(v)$-algebra 
$U = U_v(\g)$.  
We now review Caldero's proof of the existence of an $(S,<_{\lex})$-basis 
of $\O_v[G/B]$, and its natural extension to arbitrary partial flag and 
Schubert varieties \cite{C}. Since Caldero is interested in classical flag 
varieties, he works with a large base field $\CC(q)$ that allows him to 
specialize at $q=1$ and still get algebraic varieties over the complex 
numbers. We give a different version of his argument which works over $\QQ(v)$.

\paragraph
Fix as our base field $\QQ(v)$, and set $U=U_v(\g)$. We denote by 
$U^+$ and $U^-$ the algebras $U^+_v(\g)$ and $U^-_v(\g)$, respectively. 

Let $\A = \ZZ[v, v^{-1}] \subset \QQ(v)$. For each $n \in \NN$ and each 
$\alpha \in \Pi$ we write $v_\alpha$ for $v^{(\alpha, \alpha)/2}$. We also use 
the notation $[n]_\alpha = \frac{v_\alpha^{n} - v_\alpha^{-n}}{v_\alpha - 
v_\alpha^{-1}}$, and $[n]_\alpha ! = [1]_\alpha [2]_\alpha \cdots [n]_\alpha$. 
Finally, we set $F_\alpha^{(n)} = \frac{1}{[n]_\alpha!} F_\alpha^n$ and 
$E_\alpha^{(n)} = \frac{1}{[n]_\alpha!} E_\alpha^n$. 
The algebra $U$ has an $\A$-form which we denote by $U_\A$; it is the 
$\A$-subalgebra of $U$ generated by the elements of the form $F_\alpha^{(n)}, 
E_\alpha^{(n)}, K_\alpha^{\pm 1}$ for all $\alpha \in \Pi$ and all $n \geq 0$ 
\cite{Jan}*{11.1}. The algebra $U^+$, resp. $U^-$,
also has an $\A$-form which we denote by $U_\A^+$, resp. $U_\A^-$; it is 
generated by all the $E_\alpha^{(n)}$, resp. $F^{(n)}_\alpha$, with $n \geq 0$.
These $\A$-forms are compatible with the weight decomposition of $U$. 
By construction $U \cong \QQ(v) \ot_\A U_\A$, and analogous results hold for 
$U_\A^+$ and $U_\A^-$. 

For the rest of this section we fix a nonzero highest weight vector $v_\lambda 
\in V_v(\lambda)$. Setting $V_\A(\lambda) = U^-_\A v_\lambda \subset 
V_v(\lambda)$ we obtain $\A$-forms of $U$-modules. These $\A$-forms are 
compatible with the weight decompositions of the original objects, see 
\cite{Jan}*{chapter 11}. Now let $\lambda$ be a dominant integral weight and 
fix $w \in W$. The $\A$-form of the Demazure module $V_v(\lambda)_{w}$ is 
defined as $V_\A(\lambda)_{w} = U_\A^+ V_\A(\lambda)_{w\lambda}$.

\paragraph
\label{T:A-crystal-basis} 
The algebra $U_\A^-$ has a homogeneous $\A$-basis, called the \emph{canonical}
or \emph{global} basis of $U^-$, discovered independently by Lusztig and 
Kashiwara. Its construction is the subject of \cite{Jan}*{chapters 9 - 11}, 
and we will use the notation from this source to recapitulate some relevant 
facts. 

Set $A_0 \subset \QQ(v)$ to be the ring of rational functions without a pole 
at $0$. For each $\alpha \in \Pi$ define the operators $\tilde E_\alpha,
\tilde F_\alpha: U^- \to U^-$ as in \cite{Jan}*{10.2}, and let $\L(\infty)$ be 
the $A_0$-lattice generated by all elements of the form $\tilde F_{\alpha_1} 
\tilde F_{\alpha_2} \cdots \tilde F_{\alpha_r}(1)$; by definition these are 
weight elements, so setting $\L(\infty)_{-\nu} = \L(\infty) \cap U^-_{-\nu}$ 
for each $\nu \in \ZZ \Phi$ with $\nu \geq 0$, we get $\L(\infty) = 
\bigoplus_{\nu \geq 0} \L(\infty)_{-\nu}$, and furthermore each 
$\L(\infty)_{-\nu}$ is a finitely generated $A_0$-module that generates 
$U^-_\nu$ over $\QQ(v)$. 

Set
\begin{align*}
\B(\infty)_{-\nu} 
  &=  \left\{\tilde F_{\alpha_1} \tilde F_{\alpha_2} \cdots 
    \tilde F_{\alpha_r}(1) + v\L(\infty) \mid \sum_{i} \alpha_{i} = \nu\right\}
    \subset \L(\infty)_{-\nu}/v\L(\infty)_{-\nu},
\end{align*}
and set $\B(\infty) = \bigsqcup_{\nu \geq 0} \B(\infty)_{-\nu}$. Although it is
not obvious, $\B(\infty)_{-\nu}$ is a basis of $\L(\infty)_{-\nu} / 
v \L(\infty)_{-\nu}$ \cite{Jan}*{Proposition 10.11}; this is the \emph{crystal}
basis of $U^-_{-\nu}$ at $v = 0$. It turns out that each 
$b \in \B(\infty)_{-\nu}$ has a unique lift $G(b) \in \L(\infty)_{-\nu} \cap 
U_\A^-$, which is invariant under the action of certain automorphisms 
of $U$ \cite{Jan}*{Theorem 11.10 a)}. The set $G(\B)$ of all $G(b)$ with $b 
\in \B(\infty)$ is the global basis of $U_\A^-$.

Let $w \in W$, and let $w = s_{\alpha_1}\cdots s_{\alpha_r}$ be a reduced 
decomposition of $w$. Set $\B_w(\infty)$ as the set of elements in $\B(\infty)$
of the form $\tilde F_{\alpha_1}^{k_1} \tilde F_{\alpha_2}^{k_2} \cdots \tilde 
F_{\alpha_r}^{k_r}(1) + v\L(\infty)$ with $k_j \geq 0$. This set does not 
depend on the decomposition of $w$ \cite{K1}*{Proposition 3.2.5}.

\begin{Theorem*}
Let $\nu \in \ZZ \Phi$ with $\nu \geq 0$.
\begin{enumerate}[label=(\alph*)]
\item The set $\{G(b) \mid b \in \B(\infty)_{-\nu}\}$ is an $\A$-basis of 
$(U_\A^-)_{- \nu}$.

\item Let $\lambda$ be a dominant integral weight. The set $\{G(b)v_\lambda 
\mid b \in \B(\infty)_{-\nu}\} \setminus \{0\}$ is an $\A$-basis of 
$V_\A(\lambda)_{\lambda - \nu}$. Furthermore, if $G(b)v_\lambda = 
G(b')v_\lambda \neq 0$ then $b = b'$.

\item Let $w \in W$. The set $B_w(\lambda) = \{G(b)v_\lambda \mid b \in 
\B_w(\infty)\} \setminus \{0\}$ is an $\A$-basis of $V_\A(\lambda)_{w}$. 
\end{enumerate}
\end{Theorem*}
\begin{proof}
The first two items are part of \cite{Jan}*{Theorem 11.10}. The third is 
\cite{K1}*{Proposition 3.2.5 (vi)}.
\end{proof}

\paragraph
\label{q-var-bases}
We now use the global basis to produce bases for quantum Schubert varieties.
For each dominant integral weight $\lambda$ we write $\B(\lambda) = \{b \in 
\B(\infty) \mid G(b)
v_\lambda \neq 0\}$. By the previous theorem the set $\{G(b) v_\lambda \mid b 
\in \B(\lambda)\}$ is a basis of $V_v(\lambda)$, so we can take its dual 
basis. Given $b \in \B(\lambda)$, we denote by $b_\lambda^*$ the unique 
element of $V_v(\lambda)^*$ such that $b^*_\lambda(G(b') v_\lambda) = 
\delta_{b,b'}$ for all $b' \in \B(\lambda)$. Thus to each element $b \in 
\B(\lambda)$ we can associate the matrix coefficient $b_\lambda = 
c^{\lambda}_{b^*_\lambda,v_\lambda} \in C^+_v(\lambda)$, and the set 
$\{b_\lambda \mid b \in \B(\lambda)\}$ is a basis of $C^+_v(\lambda)$. Since 
the quantum flag variety is the direct sum of all these spaces with $\lambda$ 
running over all dominant integral weights, we obtain a basis of $\O_v[G/B]$ 
as defined in \ref{q-full-flag-varieties} by taking
\[
  \flagbasis = \{b_\lambda \mid b \in \B(\lambda), \lambda \in \Lambda^+\}.
\]

If $I$ is a subset of the fundamental weights, then we obtain a basis of the 
partial flag variety $\O_v[G/P_I]$ by taking
\[
  \flagbasis_I = \{b_\lambda \mid b \in \B(\lambda), \lambda \in \mathcal 
    J(I)\}.
\]
Finally, the third item of the previous theorem implies that the ideal $J^I_w$ 
defined in \ref{q-full-flag-varieties} is spanned over $\QQ(v)$ by all 
elements of the form $b_\lambda$ with $b \in \B(\lambda) \setminus 
\B_w(\infty)$. Setting $\B_w(\lambda) = \B_w(\infty) \cap \B(\lambda)$ we
obtain a basis for the quantum Schubert variety $\O_v[G/P_I]_w$ by taking the 
image of
\[
  \schubertbasis_I(w) 
    = \{b_\lambda \mid b \in \B_w(\lambda), \lambda \in \mathcal J(I) \}.
\]
in the quotient.

\paragraph
\label{Littelman-parametrizations}
Recall that the Kashiwara operators $\tilde E_\alpha, \tilde F_\alpha$ 
induce operators $\tilde E_\alpha, \tilde F_\alpha: \B(\infty) \to \B(\infty) 
\cup \{0\}$ \cite{Jan}*{10.12}. By definition the operators $\tilde E_\alpha$ 
are locally 
nilpotent as operators on $U$ \cite{Jan}*{10.2} and hence as operators on 
$\B(\infty)$, so it makes sense to set $e_{\alpha}(b) = \max \{k \in \NN \mid 
\tilde E_\alpha^k(b) \neq 0\}$ for each $b \in \B(\infty)$. We write 
$\overline E_\alpha(b) = \tilde E_{\alpha}^{e_\alpha(b)}(b)$.

To each reduced decomposition of $w \in W$ we can associate a parametrization 
of $\B_w(\infty)$, known as a string parametrization, introduced by Littelmann 
\cite{Lit} and Berenstein and Zelevinsky \cite{BZ}. If $\tilde w = s_{\alpha_1}
\cdots s_{\alpha_r}$ is the chosen decomposition then we define 
$\Lambda_{\tilde w}: \B(\infty) \to \NN^r$ by the formula
\[
  \Lambda_{\tilde w}(b) 
    = (e_{\alpha_1}(b), e_{\alpha_2}(\overline E_{\alpha_1}(b)), \ldots,
    e_{\alpha_r}(\overline E_{\alpha_{r-1}} \cdots \overline E_{\alpha_1}(b))).
\]
If $a = \Lambda_{\tilde b}$ then $\tilde F_{\alpha_1}^{a_1} \cdots
F_{\alpha_N}^{a_N} 1 = b$, so this map is injective. Now according to 
\cite{BZ}*{Proposition 3.5}, the set $S_{\tilde w} = \Lambda_{\tilde w}
(\B_w(\infty))$ is the set of integral points of a convex polyhedral cone, and 
hence by Gordan's lemma a normal affine semigroup. 

A decomposition of $w_0$, the longest word of $W$, is said to be \emph{adapted}
to $w$ if it is of the form $s_{\alpha_1} \cdots s_{\alpha_N}$ with 
$s_{\alpha_1} \cdots s_{\alpha_{\ell(w)}} = w$. For every element $w \in W$ 
there exists a decomposition of the longest word of $W$ adapted to $w$, or in 
other words the longest word of $W$ is the maximum for the weak right Bruhat 
order on $W$, see \cite{BB}*{Proposition 3.1.2}. 

\begin{Definition*}
Set $\Pi = \{\alpha_1, \ldots, \alpha_n\}$ and set $\varpi_i = 
\varpi_{\alpha_i}$. Fix $w \in W$ and fix $\tilde w_0$ a decomposition of 
$w_0$ adapted to $w$. We define
\begin{align*}
\Gamma_{\tilde w_0}:
   FB &\to \NN^{N} \times \NN^n \\
   b_\lambda & \longmapsto \Lambda_{\tilde w_0}(b) \times 
   (\langle \lambda, \alpha_1 \rangle, \ldots, \langle \lambda, \alpha_n 
   \rangle).
\end{align*}
and set $\tilde S_{\tilde w_0} = \GG_{\tilde w_0}(\flagbasis), 
\tilde S_{\tilde w_0}^I = \GG_{\tilde w_0}(\flagbasis_I)$ and 
$\tilde S_{\tilde w_0, \tilde w}^I = \GG_{\tilde w_0}(\schubertbasis_I(w))$.
\end{Definition*}
If $\Gamma_{\tilde w_0}(b_\lambda) = \Gamma_{\tilde w_0}(b'_{\lambda'})$ then 
$\lambda = \lambda'$ and $\Lambda_{\tilde w_0}(b_\lambda) = 
\Lambda_{\tilde w_0}(b'_{\lambda'})$, which implies $b_\lambda=b'_{\lambda'}$. 
Thus $\Gamma_{\tilde w_0}$ is injective and the sets $\tilde S_{\tilde w_0}, 
\tilde S_{\tilde w_0}^I$ and $\tilde S_{\tilde w_0, \tilde w}^I$ parametrize 
bases of the quantized coordinate rings of the full flag variety, the partial 
flag variety associated to the set $I$, and of the Schubert variety associated 
to $I$ and $w$, respectively. 

\begin{Lemma*}
The sets $\tilde S_{\tilde w_0}, \tilde S_{\tilde w_0}^I$ and 
$\tilde S_{\tilde w_0, \tilde w}^I$ are normal affine semigroups.
\end{Lemma*}
\begin{proof}
According to \cite{Lit}*{Proposition 1.5} $
\tilde S_{\tilde w_0}$  is the set of all $a \times (r_1, \ldots, r_n) \in 
S_{\tilde w_0} \times \NN^n$ such that
\begin{align*}
a_l 
  &\leq \left\langle \lambda - \sum_{j=1}^{N-l} a_{N-j+1} \alpha_{i_{N-j+1}}, 
  \alpha_{i_l} \right\rangle
& (1 \leq l \leq N)
\end{align*}
where $\lambda = \sum_i r_i \varpi_i$. Thus $\tilde S_{\tilde w_0}$ is the set 
of points of $S_{\tilde w_0} \times \NN^n$ that comply with these 
inequalities, and hence it is also a normal affine semigroup. 

Furthermore, $\tilde S_{\tilde w_0}^I = \Gamma_{\tilde w_0}(\flagbasis_I)$ is 
the intersection of $\tilde S_{\tilde w_0}$ with the hyperplanes defined by 
the equations $x_{N+i} = 0$ for all $i$ such that $\varpi_i \in I$, and hence 
is also a normal affine semigroup. Finally, the fact that the decomposition 
$\tilde w_0$ is adapted to $w$ implies that $\tilde S_{\tilde w_0, 
\tilde w}^I = \Gamma_{\tilde w_0}(\schubertbasis_I(w))$ is the 
intersection of $\tilde S_{\tilde w_0}^I$ with the hyperplanes defined by 
$x_i = 0$ for all $\ell(w) < i \leq N$, and hence it is also a normal affine 
semigroup. 
\end{proof}

\paragraph
\label{coproduct-coefficients}
We have just shown that the bases of quantum flag varieties and quantum 
Schubert varieties defined in \ref{q-var-bases} are parametrized by normal 
semigroups. All that is left to check is that they have the multiplicative 
property of $(S,<_{\lex})$-bases, where $<_{\lex}$ is the lexicographic order 
of $\NN^N \times \NN^n$.

Let $\lambda, \lambda' \in \Lambda^+$ and $b \in \B(\lambda), b' \in
\B(\lambda')$. Recall that $b_\lambda^*$ denotes the element in the dual basis 
of $V_v(\lambda)^*$ as defined in \ref{q-var-bases}. The product $b_\lambda 
b'_{\lambda'}$ is by definition the matrix coefficient corresponding to the 
functional $b_\lambda^* \ot b_{\lambda'}'^*$ and the vector $v_\lambda \ot 
v_{\lambda'}$ over $V(\lambda) \ot V(\lambda')$. Now the $U$-module generated 
by $v_\lambda \ot v_{\lambda'}$ is isomorphic to $V_v(\lambda + \lambda')$, so 
$b_\lambda^* \ot b_{\lambda'}'^*$ naturally induces an element in 
$C_v^+(\lambda + \lambda')$, and the product $b_\lambda b'_{\lambda'}$ is a 
linear combination of matrix coefficients in $C^+_v(\lambda + \lambda')$
\[
  b_\lambda b'_{\lambda'} 
    = \sum_{b'' \in \B(\lambda + \lambda')} c^{b''}_{b,b'} b''_{\lambda + 
    \lambda'}
\]
with $c^{b''}_{b,b'} \in \QQ(v)$. In order to show that $\flagbasis$ is an 
$(\tilde S_{\tilde \omega_0}, <_{\lex})$-basis we must show that 
$c^{b''}_{b,b'} \neq 0$ implies that $\Gamma_{\tilde \omega_0}(b'') \leq 
\Gamma_{\tilde \omega_0}(b) + \Gamma_{\tilde \omega_0}(b')$, and that if
equality holds then $c_{b,b'}^{b''}$ is nonzero. Notice that this would also 
imply that $\flagbasis_I$ is an $(\tilde S^I_{\tilde \omega_0}, <_{\lex})
$-basis and that $\schubertbasis_I(\tilde w)$ is a $(\tilde S^I_{\tilde 
\omega_0, \tilde w}, <_{\lex})$-basis.

By definition of a dual basis the scalar $c_{b,b'}^{b''}$ is the value 
$b_\lambda b'_{\lambda'}(G(b''))$. On the other hand, by the definition of the 
product of matrix coefficients
\[
  b_\lambda b'_{\lambda'}(G(b''))
    = b_\lambda^* \ot b_{\lambda'}'^*(\Delta(G(b'')) \cdot 
      v_\lambda \ot v_{\lambda'}),
\]
so we need to study the coproduct $\Delta(G(b''))$.
It follows from \cite{Jan}*{4.9 (4)} that $\Delta(F_\alpha^{(r)}) = 
\sum_{i+j=r} v_\alpha^{ji} F^{(i)}_\alpha \ot F^{(j)}_\alpha K_\alpha^{-i}$,
where $v_\alpha = v^{(\alpha,\alpha)/2}$. Thus by an argument similar to that
of \cite{Jan}*{Lemma 4.12} we see that 
\begin{align*}
\Delta(U^-_\A)_{-\nu} \subset \bigoplus_{0 \leq \mu \leq \nu} 
  (U^-_\A)_{-\mu} \ot (U^-_\A)_{-(\nu - \mu)} K_{-\mu}
\end{align*}
It follows that $\Delta(G(b''))$ is an $\A$-linear combination of terms 
$G(b_{(1)}) \ot G(b_{(2)}) K_\mu$, with $b_{(1)}, b_{(2)} \in \B(\infty)$, 
and $\mu$ the weight of $b_{(1)}$. Among all these terms there is one 
of the form $d_{b,b'}^{b''} (G(b) \ot G(b') K_{\wt(b)})$ with $d_{b,b'}^{b''} 
\in \A$, and thus $c_{b,b'}^{b''} = v^{(\wt(b), \lambda')} d_{b,b'}^{b''}$. 
Notice that unlike before, the element $d_{b,b'}^{b''}$ is independent of 
$\lambda$ and $\lambda'$.

The problem of showing that $\flagbasis$ is indeed an $(\tilde S_{\tilde w_0}, 
<_{\lex})$-basis thus reduces to showing that if $d_{b,b'}^{b''} 
\neq 0$ then $\Lambda_{\tilde w_0}(b'') \leq_{\lex} \Lambda_{\tilde w_0}(b) + 
\Lambda_{\tilde w_0}(b')$; this also implies that $\flagbasis_I$ and 
$\schubertbasis_I(\tilde w)$ are $(\tilde S_{\tilde w_0}, <_{\lex})$-bases. 
Caldero shows that this is indeed the case in \cite{C}*{Theorem 2.3}, under 
the hypothesis that $q$ is transcendental over $\CC$. We give an alternative 
proof in the following paragraphs.

\paragraph
\label{coproduct-coefficients-convex}
We fix some notation. Given a decomposition of the longest word of $W$
$\tilde w_0 = s_{\alpha_1} \cdots s_{\alpha_N}$, for each $a \in \NN^N$ we 
write $F^{(a)} = F_{\alpha_1}^{(a_1)} \cdots F_{\alpha_N}^{(a_N)}$ and 
$\tilde F^{a} = \tilde F_{\alpha_1}^{a_1} \cdots \tilde F_{\alpha_N}^{a_N}$. 
If $a$ lies in the image of $\Lambda_{\tilde w_0}$ then we write $b^a = 
\Lambda_{\tilde w_0}^{-1}(a)$. 

According to \cite{Lit}*{Proposition 10.3} the monomials $F^{(a)}$ with $a$ in 
the image of $\Lambda_{\tilde \omega_0}$ form a weight basis of $U_\A^-$. 
In fact, if we fix $\nu \in \ZZ\Phi, \nu \geq 0$, the change of basis matrix 
between Littelman's monomial basis and the global basis of $(U_\A^-)_{-\nu}$
is unipotent if we order the bases according to the lexicographic order of the
corresponding $a \in \im \Lambda_{\tilde \omega_0}$. The following lemma 
records this fact. On the other hand, we can consider a monomial of the form 
$F^{(a)}$ with $a$ outside the image of $\Lambda_{\tilde\omega_0}$;
we show that in this case the monomial is a linear combination of monomials
whose exponents are strictly larger than $a$ in the lexicographic order, or 
equivalently, elements $G(b^{a'})$ with $a' >_{\lex} a$.
\begin{Lemma*}
Fix $\tilde w_0 = s_{\alpha_1} \cdots s_{\alpha_N}$ a decomposition of the 
longest word of $W$. Let $a \in \NN^N$, and let $\nu \in \ZZ\Phi, \nu \geq 0$ 
be such that $F^{(a)} \in U^-_{-\nu}$.
\begin{enumerate}[label=(\alph*)]
\item 
\label{pre}
Let $\alpha \in \Pi$ and $r > 0$. Let $b \in \B(\infty)_{-\nu}$. Then
for each $b' \in \B(\infty)_{-\nu - r\alpha}$ with $e_\alpha(b') > r + 
e_\alpha(b)$ there exists $x_{b'} \in \A$ such that
\begin{align*}
F^{(r)}_\alpha G(b) 
  &= \qbinom{r+e_\alpha(b)}{r} G(\tilde F_{\alpha}^r b) 
  	+ \sum_{e_{\alpha}(b') > r + e_\alpha(b)} x_{b'} G(b').
\end{align*}

\item 
\label{in-image}
Suppose $a$ lies in the image of $\Lambda_{\tilde w_0}$. Then for each $a' \in 
\Lambda_{\tilde w_0} (\B(\infty)_{-\nu})$ with $a' >_{\lex} a$ there exist 
$x_{a,a'}, y_{a,a'} \in \A$ such that
\begin{align*}
F^{(a)}
  &= G(b^a) + \sum_{a' >_{\lex} a} x_{a,a'} G(b^{a'}), &
G(b^a) 
  &= F^{(a)} + \sum_{a' >_{\lex} a} y_{a,a'} F^{(a')}.
\end{align*}

\item 
\label{not-in-image}
Suppose $a$ does not lie in the image of $\Lambda_{\tilde w_0}$. Let $b = 
\tilde F^a 1 + v\L(\infty) \in \B(\infty)_{-\nu}$ and let $s = 
\Lambda_{\tilde w_0}(b)$. Then $s >_{\lex} a$, and for each $a'\in 
\Lambda_{\tilde w_0} (\B(\infty)_{-\nu})$ with $a' >_{\lex} a$ there exists 
$z_{a,a'} \in \A$ such that
\begin{align*}
	F^{(a)} &= \sum_{a' >_\lex a} z_{a,a'} G(b^{a'}).
\end{align*}
Furthermore $z_{a,s} \neq 0$.
\end{enumerate}
\end{Lemma*}
\begin{proof}
Recall from \cite{Jan}*{Lemma 11.3} that $\tilde F_\alpha^r (U_\A^-)_{-\nu} = 
\sum_{s \geq 0} F_\alpha^{(r + s)}(U_\A^-)_{-\nu + s\alpha}$, and that for 
each $u  \in (U_\A^-)_{-\nu}$ we have
\begin{align*}
\tilde F_\alpha^r u 
  = F_\alpha^{(r)} u + \tilde F_\alpha^{r+1} u'
  = F_\alpha^{(r)} u + \sum_{s \geq 1} F_\alpha^{(r+s)} u_s
\end{align*}
where $u' \in (U_\A^-)_{\nu + \alpha}$ and $u_s \in (U_\A^-)_{\nu + s\alpha}$.
Set $e = e_\alpha(b)$. As shown in the proof of \cite{Jan}*{Lemma 11.12}, in 
p. 249 below equation (1), 
\begin{align}
\label{1}
\tag{$\dagger$}
G(b) 
  &= \tilde F_\alpha^e(G(\tilde E_\alpha^e b)) 
    + \tilde F_\alpha^{e+1} u'
\end{align}
with $u' \in (U_\A^-)_{\nu + \alpha}$, so
\begin{align*}
G(b) 
  &= F_\alpha^{(e)}(G(\tilde E_\alpha^e b)) 
    + \sum_{s \geq 1} F_\alpha^{(e+s)} u_s.
\end{align*}
Multiplying this by $F_\alpha^{(r)}$ we get
\begin{align*}
F_\alpha^{(r)} G(b) 
  & = \qbinom{r+e}{r} F_\alpha^{(r + e)}(G(\tilde E_\alpha^e b)) 
    + \sum_{s \geq 1} F_\alpha^{(r+e+s)} u'_s.
\end{align*}
Using \ref{1} we get
\begin{align*}
\tilde F_\alpha^{r + e}(G(\tilde E_\alpha^e b)) 
  = \tilde F_\alpha^{r + e}(G(\tilde E_\alpha^{e+r} \tilde F_\alpha^r b))
  = G(\tilde F_\alpha^r b) + \tilde F_\alpha^{r+e+1} u''
\end{align*}
and finally
\begin{align*}
F_\alpha^{(r)} G(b) 
  & \equiv \qbinom{r+e}{r} G(\tilde F_\alpha^r b)
  \mod \tilde F_\alpha^{r+e+1} (U_\A^-)_{-\nu}.
\end{align*}
By \cite{Jan}*{Lemma 11.12} the $G(b')$ with $e_\alpha(b') > r+e$
form an $\A$-basis of $\tilde F_\alpha^{r+e+1} (U_\A^-)_{-\nu}$, so we are done
with item \ref{pre}.

As mentioned above, item \ref{in-image} is a consequence of 
\cite{Lit}*{Proposition 10.3}, which states that given a dominant integral 
weight $\lambda$ and a highest weight vector $v_\lambda \in V_\A(\lambda)$, 
then if $G(b^a) v_{\lambda} \neq 0$ there exist $x_{a,a'}$ as in the statement 
such that
\begin{align*}
F^{(a)}v_\lambda
  &= G(b^a)v_\lambda + \sum_{a' >_{\lex} a} x_{a,a'} G(b^{a'})v_\lambda.
\end{align*}
As shown in the proof of \cite{Jan}*{Theorem 10.10}, there exists a dominant
integral weight $\lambda$ such that the map $U^-_{-\nu} \to 
V_v(\lambda)_{\lambda - \nu}$ given by $G(b) \mapsto G(b)v_\lambda$ is an 
isomorphism, so the first formula is proved. This means that the (finite) 
matrix of the coefficients of the $F^{(a)} \in (U_\A^-)_{-\nu}$ in the 
global basis of $(U_\A^-)_{-\nu}$ with the order induced by the lexicographic 
order, is lower triangular with ones in the diagonal. The second formula 
follows by taking the inverse of this matrix.

Let us prove the last item. Suppose that $s <_\lex a$. Then by definition 
$s_1 = a_1, \ldots, s_j = a_j, s_{j+1} < a_{j+1}$ for some $1 \leq j \leq N$. 
This implies that 
\begin{align*}
	\tilde E_{\alpha_{j+1}}^{s_{j+1}+1} 
		\tilde E_{\alpha_j}^{s_j} \cdots \tilde E_{\alpha_1}^{s_1}(b) 
		= \tilde F_{\alpha_{j+1}}^{a_j-s_j-1} \cdots \tilde 
			F_{\alpha_N}^{a_N}1 + v\L(\infty) \neq 0,
\end{align*}
since the maps $\tilde F_\alpha$ are injective.
This contradicts the definition of $\Lambda_{\tilde w_0}(b)$, so $s 
\geq_{\lex} a$ and since $a$ is not in the image of $\Lambda_{\tilde w_0}$ 
the inequality is strict.

We now prove the following intermediate result: for each $a \in \NN^N$ we have
$F^{(a)} = z'_{a,s} G(b^s) + \sum_{a' >_{\lex} a} z'_{a,a'} F^{(a')}$, where
$s = \Lambda_{\tilde \omega_0}^{-1}(\tilde F^a(1 + v\L(\infty)))$ and 
$z'_{a,a'}, z'_{a,s}$ lie in $\A$ for all $a'$. We prove this by descending 
induction on $j = \min \{i \mid a_i \neq 0\}$, starting with the case $j = N$.
In that case item \ref{pre} implies that $\tilde F^a = G(b^a)$ and we are done.

Now let $\overline a \in \NN^N$ be the $N$-tuple given by $\overline a_j = 0$ 
and $\overline a_i = a_i$ for all $i \neq j$. By the inductive hypothesis we 
have
\begin{align*}
F^{(a)} 
  &= F_{\alpha_j}^{(a_j)}F^{(\overline a)} 
  = z'_{\overline a, \overline s} F^{a_j}_{\alpha_j} G(b^{\overline s})
    + \sum_{\overline a' > \overline a} 
      z'_{\overline a, \overline a'} F^{(a_j)}_{\alpha_j} F^{(\overline a')}.
\end{align*}
Now $F^{(a_j)}_{\alpha_j} F^{(\overline a')}$ is a scalar multiple of 
$F^{(a')}$ where $a'_j = a_j$ and $a'_i = \overline a_i$, and clearly $a' 
>_{\lex} a$. On the other hand by item \ref{pre}
\begin{align*}
F^{a_j}_{\alpha_j} G(b^{\overline s})
  &= \qbinom{a_j + e}{a_j} G(\tilde F^{a_j}_{\alpha_j} b^{\overline s})
    + \sum_{e_{\alpha_j}(b') > a_j + e_{\alpha_j}(b)} x_{b'} G(b')
\end{align*}
where $e = e_{\alpha_j}(b^{\overline s})$. Since $b^{\overline s} = 
\tilde F^{\overline a}$ we have $\tilde F_{\alpha_j}^{a_j} b^s = \tilde 
F^a(1 + v\L(\infty)) = b^s$. On the other hand, the condition $e_{\alpha_j}(b')
> a_j + e_{\alpha_j}(b)$ guarantees that $\Lambda_{\tilde \omega_0}(b') 
>_{\lex} a$ and by item \ref{in-image} each $G(b')$ is a linear combination of 
monomials $F^{(a')}$ with $a' >_{\lex} a$. This completes the proof of the 
intermediate result.

Now consider the set $D$ of all $a \notin \im \Lambda_{\tilde \omega_0}$ such 
that $F^{(a)} \in (U^-_\A)_{-\nu}$ for which the statement of 
\ref{not-in-image} fails. This set is finite and totally ordered by the 
lexicographic order, so if it is not empty then it has a maximal element $d$.
Now by the intermediate statement $F^{(d)} = z'_{d,s} G(b^s) + \sum_{d' 
>_{\lex} d} z'_{d,d'} F^{(d')}$, and since no $d'$ can be in $D$, this sum is
an $\A$-linear combination of elements $G(b^{s'})$ with $s' >_{\lex} d' 
>_{\lex} d$. This contradicts the fact that $d \in D$, and the contradiction 
arose from supposing $D$ was nonempty. Thus \ref{not-in-image} holds in all 
cases.
\end{proof}

\paragraph
Recall that we have defined $c_{b,b'}^{b''} \in \A$ as the coefficient of 
$b''_{\lambda + \lambda'}$ in $b_\lambda b_{\lambda'}$, and that we have shown 
that it equals a power of $v$ times $d_{b,b'}^{b''}$, the coefficient of 
$G(b) \ot G(b') K_{-\mu}$ in $\Delta(G(b''))$, where $-\mu$ is the weight
of $G(b)$. We have also shown above that the statement of the following 
proposition is equivalent to $\flagbasis$ being a 
$(\tilde S_{\tilde \omega_0}, <_{\lex})$-basis.  
\begin{Proposition*}
Fix $\tilde w_0 = s_{\alpha_1} \cdots s_{\alpha_N}$ a decomposition of the 
longest word of $W$. Let $b,b',b'' \in \B(\infty)$. If $d_{b,b'}^{b''} \neq 
0$ then $\Lambda_{\tilde w_0}(b'') \leq_{\lex} \Lambda_{\tilde w_0}(b) + 
\Lambda_{\tilde w_0}(b')$, and if equality holds then $d_{b,b'}^{b''}$
is a power of $v$.
\end{Proposition*}
\begin{proof}
We have already observed that $\Delta(F_\alpha^{(r)}) = \sum_{i+j = r} 
v_\alpha^{ij} F_\alpha^{(i)} \ot F_\alpha^{(j)} K^{-i}_\alpha$, and it follows 
that for each $a \in \NN^N$ we get $\Delta(F^{(a)}) = 
\sum_{t+u = a} v^{z(t,u)} F^{(t)} \ot F^{(u)}K_{\wt(F^{(t)})}$, with $z(t,u) 
\in \ZZ$. By item \ref{in-image} of the previous lemma 
\begin{align*}
\Delta(G(b^a))
  &= \Delta(F^{(a)}) + \sum_{a' >_{\lex} a} y_{a,a'} \Delta(F^{(a')}) \\
  &= \sum_{t+u = a} v^{z(t,u)} F^{(t)} \ot F^{(u)}K_{\wt(F^{(t)})} 
  + \sum_{t+u >_\lex a} y_{a,t+u} v^{z(t,u)} F^{(t)}
    \ot F^{(u)}K_{\wt(F^{(t)})},
\end{align*}
and using item \ref{not-in-image} of the lemma we get that the element in the 
last display equals
\begin{align*}
\sum_{t+u = a} v^{z(t,u)} G(b^t) \ot G(b^u) K_{\wt(b^t)}
    + \sum_{t+u >_{\lex} a} y'_{t,u} G(b^t) \ot G(b^u) K_{\wt(b^t)},
\end{align*}
where $t,u$ run over the image of $\Lambda_{\tilde w_0}$, and $z(t,u) \in \ZZ$
and $y'_{t,u} \in \A$ for each such pair $(t,u)$. The result follows by taking 
$b^a = b''$.
\end{proof}

\begin{Remark*}
It is possible to take an alternative approach, and define bases for quantum 
flag and Schubert varieties as in paragraph \ref{q-var-bases} starting with 
the monomial basis instead of the global basis, and in this case we also 
obtain an $(S,<_{\lex})$-basis. This is the approach taken by Fang,
Fourier and Littelmann in \cite{FFL2}, where they obtain many different 
monomial bases for (classical) enveloping algebras and hence many different 
degenerations for a larger class of varieties (in the commutative case). The 
price to pay is that one loses control over the semigroup parameterizing the 
basis i.e. the exponents of the monomial basis. In general it is not 
known whether this semigroup is affine (though in some cases it is known that 
it is not normal, see the aforementioned article). In our case this is 
guaranteed by the fact that this semigroup is the same as that arising from
the string parametrization, which is known to be affine. Thus even in the
alternative approach the relation between the monomial basis and the canonical
basis is essential. 
\end{Remark*}

\paragraph
\label{Schubert-deg}
Now let $\k$ be an arbitrary field with $\operatorname{char} \k > 2$, or 
$\operatorname{char} \k > 3$ if the root system of $\g$ has an irreducible 
component
of type $G_2$, and let $q \in \k^\times$ be a nonroot of unity. There is a 
morphism $\A \to \k$ induced by the assignation $v \mapsto q$, which makes 
$\k$ into an $\A$-bimodule. There is an algebra map $\kappa: U^-_q(\g) \to \k 
\ot_\A U_\A^-$, given by sending $F_{\alpha_i} \in U_q^-(\g)$ to $1 \ot_{\A} 
F_{\alpha_i}$. This map is obviously surjective, and it respects the weight 
decomposition of both algebras. By \cite{Jan}*{8.24 Remark (3)}, the dimension
of the weight components of both algebras are given by the Kostant partition
function, and hence they are equal. Thus $\kappa$ is an isomorphism. 

The map $\k \ot_\A V_\A(\lambda) \to V_q(\lambda)$ is $U^-_q(\g)$-linear 
and sends highest weight vectors to highest weight vectors, so it is an 
isomorphism and the global basis of $V_\A(\lambda)$ maps to a basis of 
$V_q(\lambda)$, which we also call the global basis of $V_q(\lambda)$. Also, 
if we set $C_\A^+(\lambda)$ as the $\A$-span of $b_\lambda$ for $b \in 
\B(\lambda)$ we get an isomorphism $\k \ot_\A C_\A^+(\lambda) \mapsto 
C_q^+(\lambda)$, with the image of the $b_\lambda$ forming the dual basis of 
the global basis of $V_q(\lambda)$. Thus if we set $\O_\A[G/P_I]_w$ to be the 
$\A$-span of $SB_I(w)$ inside $\O_v[G/P_I]_w$, we get that the natural map 
$\k \ot_\A \O_\A[G/P_I]_w \to \O_q[G/P_I]_w$ is an isomorphism.

\begin{Theorem*}
Let $\k$ be any field and let $q \in \k^\times$ be a nonroot of unity. Let $I 
\subset P$ be a set of fundamental weights, let $w \in W$, and let $\tilde w_0$
be a reduced decomposition of $w_0$ adapted to $w$.

The quantum Schubert variety $\O_q[G/P_I]_w$ degenerates to a quantum affine 
toric variety with associated semigroup $\tilde S^I_{\tilde w_0, \tilde w}$. 
In particular any quantum Schubert variety has property $\chi$, finite local 
dimension, the AS-Cohen-Macaulay property, and is a maximal order in its 
skew-field of fractions.
\end{Theorem*}
\begin{proof}
For each dominant integral weight $\lambda$ and each $a \in 
\Lambda_{\tilde w_0}(\B(\lambda))$, we denote by $b^a_\lambda$ the element $1 
\ot_\A b_\lambda \in \k \ot_\A \O_v[G/P_I]_w$ where $b = 
\GG_{\tilde w_0}^{-1}(a)$. With this notation, it follows from 
\ref{coproduct-coefficients-convex} that for each pair of dominant integral 
weights $\lambda, \lambda'$ and each $a \in \Lambda_{\tilde w_0}(\B(\lambda)), 
a' \in \Lambda_{\tilde w_0}(\B(\lambda')), a'' \in 
\Lambda_{\tilde w_0}(\B(\lambda + \lambda'))$ such that $a'' \geq_{\lex} a + 
a'$, there exists $x_{a''} \in \A$ such that 
\begin{align*}
b^a_{\lambda} b^{a'}_{\lambda'} = \sum_{a'' \geq a+a'} x_{a''} 
  b^{a''}_{\lambda + \lambda'},
\end{align*}
with $x_{a+a'}$ a power of $q$. This implies that the basis 
$\{1 \ot b \mid b \in \schubertbasis_I(w)\}$ is a $(\tilde S^I_{\tilde w_0, 
\tilde w}, \leq_{\lex}~)$ basis. Thus by Proposition 
\ref{S-ordered-basis-degeneration} we get the degeneration result.

Since $\tilde S^w_{w_0, I}$ is a normal semigroup, we know by Proposition 
\ref{properties-of-qatv} that the associated graded ring of the quantum 
Schubert variety has property $\chi$, finite local dimension and the 
AS-Cohen-Macaulay property, which $\O_q[G/P_I]_w$ inherits by Theorem 
\ref{transfer}. 
Also by Proposition \ref{properties-of-qatv}, a quantum affine toric variety 
whose underlying semigroup is normal is a maximal order in its ring of 
fractions, and it follows from \cite{Mau}*{Chapitre IV, Proposition 2.1 and 
Chapitre V, Corollaire 2.6} that $\O_q[G/P_I]_w$ is also a maximal order. 
\end{proof}


\vskip .5cm

\noindent
Laurent RIGAL, \\
Universit\'e Paris 13, Sorbonne Paris Cit\'e, LAGA, UMR CNRS 7539, 99 avenue 
J.-B. Cl\'ement,
93430 Villetaneuse,  France; e-mail: rigal@math.univ-paris13.fr\\

\noindent
Pablo ZADUNAISKY, \\
Universidad CAECE, Departamento de Matem\'aticas.
Av. de Mayo 866 - Buenos Aires, Argentina. e-mail: pzadub@dm.uba.ar\\

\end{document}